\documentclass[11pt]{amsart}
\usepackage[T1]{fontenc}
\usepackage[utf8]{inputenc}
\usepackage{lmodern}
\usepackage{amsmath,amssymb,amsthm,mathtools}
\usepackage{enumitem,array,booktabs,needspace}
\usepackage[table]{xcolor}
\usepackage[a4paper,margin=2.8cm]{geometry}
\usepackage[hypertexnames=false,hidelinks]{hyperref}
\definecolor{lightgray}{gray}{0.96}
\definecolor{headergray}{gray}{0.91}
\setlist{itemsep=2pt,topsep=5pt,parsep=1pt}
\numberwithin{equation}{section}
\theoremstyle{plain}
\newtheorem{theorem}{Theorem}[section]
\newtheorem{proposition}[theorem]{Proposition}
\newtheorem{lemma}[theorem]{Lemma}
\newtheorem{corollary}[theorem]{Corollary}
\theoremstyle{definition}
\newtheorem{definition}[theorem]{Definition}

\theoremstyle{plain}
\newtheorem{example}[theorem]{Example}
\theoremstyle{plain}
\newtheorem{remark}[theorem]{Remark}
\DeclareMathOperator{\diag}{diag}

\DeclareMathOperator{\supp}{supp}

\allowdisplaybreaks[1]
\title[The ``para-classical'' orthogonal polynomials]
{The manufacture of examples II: the ``para-classical'' orthogonal polynomials}

\author{K. Castillo}
\address{CMUC, Department of Mathematics, University of Coimbra,
3000-143 Coimbra, Portugal}
\email{math@keniercastillo.com}
\author{G. Gordillo-N\'u\~nez}
\address{CMUC, Department of Mathematics, University of Coimbra,
3000-143 Coimbra, Portugal}
\email{up202310693@up.pt}
\date{}
\hypersetup{pdftitle={The manufacture of examples II: the "para-classical" orthogonal polynomials},
pdfauthor={K. Castillo and G. Gordillo-Nunez}}
\begin{document}

\begin{abstract}
Under the stated regularity assumptions, the full orthogonality
functionals of the five finite para-families considered here are classical
for the respective affine, quadratic, $q$-quadratic, $q$-linear
and alternating calculi.
For each of the para--Krawtchouk, para--Racah, $q$-para--Racah and
$q$-para--Krawtchouk families, two finite solutions of one global
structural equation have $N+1$ nodes altogether and common moments
through degree $N-1$.
Their relative mass changes only the central recurrence coefficients.
We prove regularity for every normalised combination with both component
masses non-zero and recover the full published functionals and recurrences.
The alternating
construction identifies para--Bannai--Ito within the same structural
framework. For the four non-alternating families, equal positive combinations give
persymmetric Jacobi matrices; the known spectral gap criterion
then decides perfect state transfer. Explicit parameter correspondences
locate the classical components and resonant limits of the full systems.
A positive seven-point example shows how further named families can
be manufactured without a new governing principle. Counterexamples
to mathematical assertions in three of four recent preprints expose
the consequences of insufficient hypotheses for finite representations
and physical models. In an era of AI-assisted
symbolic production, we call for a standard that rewards structural
explanation and verification rather than the proliferation of examples.
\end{abstract}

\maketitle

\section{Introduction}\label{sec:introduction}

We had set this manuscript aside in the wake of recent advances in
artificial intelligence. We expected AI tools, given the structural
theory already developed in our earlier work, to recover these finite
families with little further guidance if a genuine need arose.
However, four preprints
\cite{BC26,BCLMRV26K,BCLMRV26R,BCLMRV26B}, submitted to arXiv on
24 September 2026 and comprising 126 pages in their first versions,
have prompted us to
return to the manuscript, make the structural explanation explicit
and examine their claims.

The title continues \emph{The manufacture of examples: the
$(-1)$-classical orthogonal polynomials} \cite{CG26b}, now accepted
for publication in \emph{Res. Math. Sci.}, a journal we hold in the
highest regard. Here ``para-classical'' refers to the families
conventionally called para--Krawtchouk, para--Racah, $q$-para--Racah,
$q$-para--Krawtchouk and para--Bannai--Ito. Under the stated assumptions,
their full orthogonality functionals are classical for fixed polynomial
calculi. The quotation marks are deliberate: in the structural sense
established below, these are classical families in disguise. The prefix
introduces no new category of classicality.

The explanation rests on the theory developed in
\emph{Half-step-invariant sets, admissible maps, and classical
orthogonal polynomials} \cite{C26}.
A half-step-invariant domain and an admissible map determine the
polynomial operators; a global structural equation defines
classicality. The five families studied here occupy only a small part of
the classical systems associated with admissible maps.
The framework of \cite[Sections~7--8]{C26} treats finite
and infinite orthogonality, supplies regularity criteria and recurrence
formulas, and constructs further families and their functionals.
It also recovers finite measures: when a terminal polynomial annihilates
the functional and has simple zeros, its zeros and the recurrence
determine the support and weights
\cite[Proposition~7.14 and Example~7.15]{C26}.
We specialise the presentation to these families and the structure
governing them, so that their construction, free parameter and complete
recurrence can be understood directly.

For each of the four non-alternating families, endpoint zeros in one
structural equation close two finite progressions with $N+1$ nodes
altogether. Cancellation between neighbouring masses determines their
weights up to a constant. Normalising each component gives two
mass-one solutions. Their moments through degree $N-1$ coincide;
the moment of degree $N$ determines all later moments. The relative
mass parametrises this freedom. The shared moments fix the first half
of the recurrence, the relative mass changes only its central
coefficients, and the remaining entries follow by reflection.
Regularity holds for every complex combination parameter other than
$0$ and $1$, under the stated component assumptions. Positivity is
required later for the physical matrices. The alternating case uses
the same structural definition through a different construction.

This wave follows more than a decade of publications. Earlier studies
of para--Krawtchouk, para--Racah, $q$-para--Racah,
$q$-para--Krawtchouk and para--Bannai--Ito appeared
in \emph{J. Phys. A: Math. Theor.}~\cite{VZ12},
\emph{J. Math. Anal. Appl.}~\cite{LVZ16,LVZ18},
\emph{J. Math. Phys.}~\cite{BGVZ21} and
\emph{SIGMA}~\cite{PVZ23}, respectively. Direct algebraic and
physical developments appeared in
\emph{SIGMA}~\cite{GVZ13},
\emph{Commun. Math. Phys.}~\cite{GIVZ14},
\emph{Ann. Phys.}~\cite{GVZ16},
\emph{J. Phys. A: Math. Theor.}~\cite{LVZFR16,TVZ17,GTVZ20}
and \emph{J. Math. Phys.}~\cite{BGVZ22}.
These twelve earlier articles occupy 211 pages in their published
journal versions.
Other related developments are recorded in
\cite{GTVZ17,BGBGVZ22,SVZ22}.
Recurrence-based characterisations and the para--Krawtchouk case have
already been reassessed within the general classical framework
\cite[Sections~2 and~4]{CG25}.
Our own critical work has encountered serious difficulties in gaining
acceptance. The referee report on the companion paper and our response
allow readers to examine one such exchange for themselves.\footnote{See the
referee report (\url{https://keniercastillo.com/manufacture-examples-report.pdf})
and the
authors' response (\url{https://keniercastillo.com/manufacture-examples-response.pdf}).}
The report's opening objection appeals to a handbook's standing and
to a supposed majority convention to reject our use of orthogonality.
Neither refutes the regular-functional definition adopted here.
An objection must address that definition and its consequences.
\enlargethispage{4pt}

The earlier papers provide the historical and mathematical setting.
Our examination of specific claims in Section~\ref{sec:discussion}
concerns the four September~2026 preprints. The broader question is
what mathematical work warrants publication when AI-assisted tools
make symbolic constructions easier to produce and examine. A further
construction must establish what it explains and adds; neither
professional standing nor accumulated publications can replace that
argument. We impose that obligation first on our own work.

The term ``bi-lattice'' describes a support. It supplies no polynomial
calculus and no criterion of classicality; treating it as a structural
classification is misleading. Its role in classicality is determined by
the half-step-invariant domain, admissible map and global equation.

Consider the odd para--Krawtchouk system, with $N=2j+1\ge3$ and
$0<\gamma<2$. Its nodes are \cite[Eq.~(3.1)]{VZ12}
\[
 y(k)=k+\frac{\gamma-1}{2}\{1-(-1)^k\},
 \quad 0\le k\le N.
\]
Thus $y(2k)=2k$ and $y(2k+1)=2k+\gamma$. With its published
weights $\omega_k$ and $X(s)=2s$, its orthogonality functional
can equally be written as
\[
 \langle\mathbf u,p\rangle
 =\sum_{k=0}^j\omega_{2k}p(X(k))
  +\sum_{k=0}^j\omega_{2k+1}p(X(k+\gamma/2)).
\]
The two progressions are values of one affine map. Neither the
functional nor its polynomials have changed. More generally, any
finite nodes $x_k$ can be written as $X_h(s_k)$, with
$X_h(s)=hs+v$ and $s_k=(x_k-v)/h$, for any $h\ne0$.
A support parametrisation therefore establishes no new classical
structure. Its relation to the half-step-invariant domain and the
admissible map is what defines the polynomial operators.

For the affine map $X_h$ and a polynomial $p$, the half-step divided difference is
\[
 (D_hp)(x)=\frac{p(x+h/2)-p(x-h/2)}h,\quad
 D_hx^3=3x^2+\frac{h^2}{4}.
\]
Nor does the interlaced listing itself provide the required
calculus. Extend $y$ by the same formula to all integer indices.
For a polynomial $f$, put
\[
 (\mathcal S_yf)(y(k))=\frac{f(y(k+1))+f(y(k-1))}{2},
 \quad
 (\mathcal D_yf)(y(k))=\frac{f(y(k+1))-f(y(k-1))}{y(k+1)-y(k-1)}.
\]
Then
\[
 (\mathcal S_yx)(y(k))=y(k)+(\gamma-1)(-1)^k,
 \quad
 (\mathcal D_yx^2)(y(k))=2y(k)+2(\gamma-1)(-1)^k.
\]
For $\gamma\ne1$, these fail the required polynomial degree
bounds. An affine polynomial agreeing with the first expression at
$0,2$ must be $x+\gamma-1$, which disagrees at the odd nodes.
The original para--Krawtchouk paper already recognises classicality
on a uniform grid: its equation uses $x\mapsto x\pm2$
\cite[Eqs.~(3.14)--(3.15) and the following discussion]{VZ12}.
The apparent order changes when that equation is written in the
interlaced numbering; the underlying affine calculus does not.

The exclusions of classicality in the para--Racah and
$q$-para--Racah papers concern the usual non-degenerate spectral
setting; the latter explicitly refers to Leonard duality
\cite{LVZ16,LVZ18}. The structural definition of \cite{C26}, also used
in \cite{CG26b}, admits finite regular functionals and resonant
eigenvalues. We verify its global structural equation for the full
published functionals.

Section~\ref{sec:manufactured} constructs a positive seven-point example
whose parameter changes precisely the central recurrence coefficients.
Section~\ref{sec:discussion} examines the four recent papers.
Appendix~\ref{app:classical-identifications} gives explicit connections
with the named classical families in \cite{KLS10}, distinguishing
component substitutions from resonant limits of the full systems.
The latter must retain the parameter determining the central recurrence
coefficients.

Table~\ref{tab:structural-location} states the maps and pairs we
shall verify. The first four use the usual Nikiforov--Uvarov
calculi \cite{NU83,NU84}; the last uses the alternating calculus
of \cite{C26,CG26b}. All parameters are those of the original
families. The construction requires no additional definition of
orthogonality.

% The table is inserted from the original manuscript by assemble_revision.py.
\begin{table}[htbp]
\centering
\small
\renewcommand{\arraystretch}{1.3}
\setlength{\tabcolsep}{5pt}
\begin{tabular}{@{}p{.24\linewidth}p{.72\linewidth}@{}}
\toprule
\rowcolor{gray!18}
Family and map & structural pair\\[7pt]
\midrule
Para--Krawtchouk\par
$X(s)=2s$
&
$\displaystyle
 \begin{aligned}[t]
  \phi(x)&=-\frac{x^2-(m+\gamma)x+2r(\gamma/2+\ell)}{m},\\[7pt]
  \psi(x)&=x-\frac{2r(\gamma/2+\ell)}{m}.
 \end{aligned}$
\\[7pt]
\midrule
\rowcolor{gray!7}
Para--Racah\par
$X(s)=-s^2$
&
$\displaystyle
 \begin{aligned}[t]
  \phi(x)&=-\frac1m
  \bigl[x^2-\{r\ell-a(a+r)-c(c+\ell)\}x+ac(a+r)(c+\ell)\bigr],\\[7pt]
  \psi(x)&=x+\frac{r c(c+\ell)+\ell a(a+r)}{m}.
 \end{aligned}$
\\[7pt]
\midrule
$q$-para--Racah\par
$\displaystyle X(s)=\frac{q^s+q^{-s}}2$
&
$\displaystyle
 \begin{aligned}[t]
  \phi(x)&=\frac{t-t^{-1}}{4(E_4-1)}
  \bigl[2(1+E_4)x^2-(E_1+E_3)x+E_2-1-E_4\bigr],\\[7pt]
  \psi(x)&=x+\frac{E_1-E_3}{2(E_4-1)}.
 \end{aligned}$
\\[7pt]
\midrule
\rowcolor{gray!7}
$q$-para--Krawtchouk\par
$X(s)=q^s$
&
$\displaystyle
 \begin{aligned}[t]
  \phi(x)&=\frac{t-t^{-1}}{2(q^{-m}-1)}
  \bigl[(1+q^{-m})x^2
  -\{\Delta+1+q^{-m}(\Delta q^r+q^\ell)\}x+2\Delta\bigr],\\[7pt]
  \psi(x)&=x+
  \frac{\Delta+1-q^{-m}(\Delta q^r+q^\ell)}{q^{-m}-1}.
 \end{aligned}$
\\[7pt]
\midrule
Para--Bannai--Ito\par
$X(s)=e^{\pi i s}$
&
$\displaystyle
 \begin{aligned}[t]
  (\phi,\psi)&=(0,x-A)
       &&\text{(complementary)},\\[7pt]
  (\phi,\psi)&=((x-\rho)(x-A),0)
       &&\text{(Bannai)}.
 \end{aligned}$
\\[7pt]
\bottomrule
\end{tabular}

\caption[Admissible maps and structural pairs for the finite para-families]{Admissible maps and structural pairs for the finite para-families.\label{tab:structural-location} In the first four rows, $N+1$ is the number of masses and \begin{equation}\label{eq:r-ell}
 r=\lfloor N/2\rfloor,\quad
 \ell=\lceil N/2\rceil-1,\quad
 m=r+\ell=N-1,\quad N\ge2.
\end{equation}
In the $q$-cases, $t=q^{1/2}$, where $q\ne0$ is not a root
of unity. For the $q$-para--Racah pair, we put
\begin{equation}
 \begin{aligned}
  E_1&=a+c+a^{-1}q^{-r}+c^{-1}q^{-\ell},\\[7pt]
  E_2&=ac+q^{-r}+q^{-\ell}
       +(a/c)q^{-\ell}+(c/a)q^{-r}+(ac)^{-1}q^{-m},\\[7pt]
  E_3&=q^{-m}(a^{-1}+c^{-1}+a q^r+c q^\ell),
  \quad E_4=q^{-m}.
 \end{aligned}
\end{equation}
The letters $a,c,\gamma,\Delta$ are the parameters of the
corresponding families. The parameters $\rho$ and $A$ are defined in \eqref{eq:BI-parameters}.}
\end{table}

\needspace{6\baselineskip}
Let $\mathcal P=\mathbb C[x]$ and let $\mathcal P'$ be its
algebraic dual, consisting of complex linear functionals. The pairing
$\langle\mathbf u,p\rangle$ is bilinear, and
$\langle\delta_z,p\rangle=p(z)$ defines a mass. Following
\cite{C26}, take a domain $U$ with $U\pm\tfrac12\subseteq U$ and
a map $X:U\to\mathbb C$, and put
\[
 Y(s)=X(s+\tfrac12),\quad Z(s)=X(s-\tfrac12).
\]
We use admissible maps: $X(U)$ is infinite, $Y-Z$ is non-zero
on $U$, and
\begin{equation}\label{eq:DS}
 (Dp)(X(s))=\frac{p(Y(s))-p(Z(s))}{Y(s)-Z(s)},
 \quad
 (Sp)(X(s))=\frac{p(Y(s))+p(Z(s))}{2}
\end{equation}
define polynomial operators, with $D$ lowering degree by at least
1 and $S$ not increasing it. These are operators on all of
$\mathcal P$. The finite support of a functional need not be
contained in $X(U)$. For the explicit maps below, the notation
$X,Y,Z$ also denotes their expressions outside $U$ wherever
those expressions are defined. Such evaluations parametrise
the masses; the calculus is still determined on the admissible
domain. At a removable singularity of a displayed quotient,
evaluations use the polynomial extension of the operator.
For Table~\ref{tab:structural-location}, explicit admissible sets are
$U=\mathbb C$ in the affine, $q$-linear and alternating cases, and
$U=\mathbb C\setminus\tfrac12\mathbb Z$ for $X(s)=-s^2$.
In the $q$-cases fix a logarithm $L_q$ of $q\ne0$, put
$q^s=e^{L_qs}$, $t=e^{L_q/2}$, and exclude roots of unity.
For the $q$-quadratic map use
$U=\mathbb C\setminus\{\pi i k/L_q+n/2:k,n\in\mathbb Z\}$.
These sets are half-step-invariant, $X(U)$ is infinite and $Y-Z\ne0$.
For each map, $Y+Z$ and $YZ$ are polynomials in $x=X(s)$
of degrees at most 1 and 2; the symmetric monomial quotient
and average in \eqref{eq:DS} therefore have the required degrees.

For $f,p\in\mathcal P$ and $\mathbf u\in\mathcal P'$, define
multiplication and the dual operators by
\[
 \langle f\mathbf u,p\rangle=\langle\mathbf u,fp\rangle,
 \quad
 \langle\mathbf D\mathbf u,p\rangle=-\langle\mathbf u,Dp\rangle,
 \quad
 \langle\mathbf S\mathbf u,p\rangle=\langle\mathbf u,Sp\rangle.
\]
These pairings define elements of $\mathcal P'$ because $fp,Dp,Sp$
are polynomials. Write $\mu_k=\langle\mathbf u,x^k\rangle$.
For $p=\sum_{k=0}^d p_kx^k$ one has
$\langle\mathbf u,p\rangle=\sum_{k=0}^d p_k\mu_k$.
Every sum is finite: arbitrary moment sequences define functionals
without convergence or growth assumptions. Thus the dual actions and
identities in $\mathcal P'$ require no topology or continuity hypothesis.
Put $\Delta_n=\det(\mu_{i+j})_{i,j=0}^n$. Following \cite{C26},
for any integer $M\ge1$ we call $\mathbf u$ \emph{regular of order $M$} if
\[
 \Delta_n\ne0,\quad 0\le n<M.
\]
Equivalently, it has unique monic orthogonal polynomials
$P_0,\ldots,P_{M-1}$ with non-zero squared norms. Orthogonality
for a monic polynomial of degree $n\ge1$ is a linear system
with determinant $\Delta_{n-1}$; its squared norm is
$\Delta_n/\Delta_{n-1}$, with $\Delta_{-1}=1$. No condition
on $\Delta_M$ is imposed. If its support consists of $M$ distinct
points, their monic product is the terminal polynomial $P_M$:
it annihilates the functional and has zero squared norm.
Over $\mathbb C$, distinct nodes and non-zero weights alone do not
ensure all the required Hankel determinants. Positivity is a
separate property.

\begin{definition}\label{def:classical}
A functional regular of order at least 3 is classical with
respect to $X$ if it satisfies the structural equation
\begin{equation}\label{eq:structural}
 \mathbf D(\phi\mathbf u)=\mathbf S(\psi\mathbf u),
\end{equation}
where $\phi,\psi$ have degrees at most 2 and 1, respectively,
and are not both zero.
\end{definition}

This is \cite[Definition~6.6]{C26}. We call $(\phi,\psi)$ the
structural pair. Multiplying both entries by the same non-zero
constant preserves the equation; either entry may be zero.
The additional non-vanishing assumptions of the recurrence theorems
are not part of this definition. A one- or two-point component may be regular
of order 1 or 2 without meeting the order required for classicality;
the full functional is the object classified here.

This definition follows the structural approach of Bochner and
Nikiforov--Uvarov. In the ordinary differential setting,
$D=d/dx$ and $S=\mathrm{id}_{\mathcal P}$ give the structural
equation and Bochner's differential operator.
For the admissible maps considered here, put $L=\phi D^2+\psi SD$.
The structural equation gives formal symmetry:
\[
 \langle\mathbf u,(Lp)f\rangle=\langle\mathbf u,p(Lf)\rangle,
 \quad p,f\in\mathcal P.
\]
Regularity and the degree bounds then give
$LP_n=\lambda_nP_n$, for suitable eigenvalues $\lambda_n$,
throughout the regular degree range
\cite[Theorem~9.1 and Corollary~9.4]{C26}.
Finite resonance allows repeated eigenvalues; in the alternating
regime the spectrum may only distinguish parity, and $L$ may even
vanish \cite[Remark~9.5]{C26}. The structural equation remains
the condition defining classicality.

The monic three-term recurrence is
\begin{equation}\label{eq:TTRR}
 P_{n+1}(x)=(x-B_n)P_n(x)-C_nP_{n-1}(x),
 \quad P_{-1}=0,\quad P_0=1,\quad C_0=0.
\end{equation}
For a finite system we retain $B_0,\ldots,B_N$ and
$C_1,\ldots,C_N$, with all $C_n\ne0$. The last row defines
the terminal polynomial $P_{N+1}$.

We first construct the finite solutions and prove regularity and
the recurrence identities in Section~\ref{sec:two-progressions}.
Section~\ref{sec:families} identifies the published functionals;
Section~\ref{sec:PST} adds the spectral gap condition for transfer
at a prescribed time.

\section{The structural equation and finite masses}
\label{sec:two-progressions}

With the calculus fixed by half-step invariance and admissibility,
we now construct finite functionals satisfying the global structural
equation. Their normalisation and regularity require separate checks.

We separate three arguments. Endpoint cancellation constructs
finite solutions of the global equation. For the four non-alternating
pairs, the moment equations leave one moment free after normalisation.
Two regular components with the required common moments then
determine the recurrence of their combination.
This last argument uses neither classicality nor a particular map.
It explains the four non-alternating para-families through hypotheses
independent of their names. The alternating case is treated in
Section~\ref{sec:para-BI}.

\subsection{Finite sums of masses}
Fix the half-step-invariant domain $U$, the admissible map $X$
and its operators $D,S$ as in Section~\ref{sec:introduction}.
Let $\phi,\psi$ have degrees at most 2 and 1, respectively,
and suppose they are not both zero.
For a finite arithmetic progression in $U$, consider
\begin{equation}\label{eq:finite-ansatz}
 \mathbf u=\sum_{k=0}^{r}\omega_k\delta_{X(a+k)}.
\end{equation}
The structural equation for this functional is equivalent to
\[
 0=\langle\mathbf u,\phi Dp+\psi Sp\rangle
   =\sum_{k=0}^{r}\omega_k(\phi Dp+\psi Sp)(X(a+k)),
 \quad p\in\mathcal P.
\]
The half-step construction makes neighbouring evaluations coincide.
To express their coefficients, put
\begin{equation}\label{eq:lambdas}
 \lambda_-(s)=\frac{\phi(X(s))}{Y(s)-Z(s)}
                      -\frac{\psi(X(s))}{2},\quad
 \lambda_+(s)=\frac{\phi(X(s))}{Y(s)-Z(s)}
                      +\frac{\psi(X(s))}{2}.
\end{equation}
The definitions of $D,S$ in \eqref{eq:DS} give
\[
 (\phi Dp+\psi Sp)(X(s))
 =\lambda_+(s)p(Y(s))-\lambda_-(s)p(Z(s)).
\]
Since $Y(a+k)=Z(a+k+1)$, the two contributions from
neighbouring masses evaluate $p$ at the same point.
Collecting them, we obtain
\begin{align*}
 \langle\mathbf u,\phi Dp+\psi Sp\rangle
 ={}&-\omega_0\lambda_-(a)p(Z(a))
     +\omega_r\lambda_+(a+r)p(Y(a+r))\\[7pt]
 &+\sum_{k=0}^{r-1}
  \{\omega_k\lambda_+(a+k)-\omega_{k+1}\lambda_-(a+k+1)\}
          p(Y(a+k)).
\end{align*}
The two endpoint coefficients and every interior coefficient
therefore vanish if
\[
 \lambda_-(a)=0,\quad \lambda_+(a+r)=0,\quad
 \omega_{k+1}\lambda_-(a+k+1)=\omega_k\lambda_+(a+k)
 ,\quad 0\le k<r.
\]
When the interior factors $\lambda_+(a+k)$ and
$\lambda_-(a+k+1)$ are non-zero, their ratios determine
all weights from $\omega_0$. The cancellation is valid for every
polynomial, regardless of its degree. It therefore proves the
identity in $\mathcal P'$, rather than a truncated moment test.

\begin{proposition}\label{lem:one-progression}
Suppose that the nodes $X(a),\ldots,X(a+r)$ are distinct,
\begin{equation}
 \lambda_-(a)=0,\quad \lambda_+(a+r)=0,
\end{equation}
and $\lambda_+(a+k)\lambda_-(a+k+1)\ne0$ for $0\le k<r$.
Then, for any $\omega_0\ne0$, the weights
\begin{equation}
 \omega_k=\omega_0\prod_{i=0}^{k-1}
       \frac{\lambda_+(a+i)}{\lambda_-(a+i+1)},
 \quad 0\le k\le r,
\end{equation}
give a solution of \eqref{eq:structural} of the form
\eqref{eq:finite-ansatz}. 

If $Z(a),Y(a),\ldots,Y(a+r)$
are also distinct, the structural equation forces the endpoint
conditions and the neighbouring cancellations for every solution
with non-zero weights. Under the stated non-zero conditions,
the weights are therefore unique up to their common factor.
\end{proposition}
\begin{proof}
The product gives
$\omega_{k+1}\lambda_-(a+k+1)=\omega_k\lambda_+(a+k)$,
so every coefficient in the collected sum is zero. Conversely,
evaluations at distinct points are independent: a polynomial can
be chosen to be one at any one of them and zero at the others.
Hence each coefficient must vanish.
\end{proof}

All evaluations must be defined for the chosen calculus, with
removable singularities interpreted as in Section~\ref{sec:introduction}.
For the explicit maps below, that continuation convention
also allows progressions outside $U$: the displayed identity and
every half-step evaluation used in the cancellation must be
defined there. Their nodes need not belong to $X(U)$.

Normalisation requires a non-zero sum of the constructed weights;
non-zero individual weights on distinct complex nodes do not
imply this. Regularity also remains to be established. The endpoint
argument proves the structural identity without a Hankel assumption.

\begin{example}[The two para--Krawtchouk sums]\label{ex:PK-masses}
For the affine map and pair in Table~\ref{tab:structural-location},
\[
 \lambda_-(s)=-\frac{2s(s-\gamma/2)}m,\quad
 \lambda_+(s)=-\frac{2(s-r)(s-\gamma/2-\ell)}m.
\]
The two progressions start at $0,\gamma/2$ and end at
$r,\gamma/2+\ell$. For $0<\gamma<2$ the nodes are distinct
and all interior weight ratios are positive:
\begin{equation}\label{eq:PK-ratios}
 \begin{aligned}
 &\frac{(k-r)(k-\gamma/2-\ell)}{(k+1)(k+1-\gamma/2)},
 \quad 0\le k<r,\\[7pt]
 &\frac{(k-\ell)(k+\gamma/2-r)}{(k+1)(k+1+\gamma/2)},
 \quad 0\le k<\ell.
 \end{aligned}
\end{equation}
With empty products equal to 1, write $(b)_k=\prod_{i=0}^{k-1}(b+i)$
and $(b_1,\ldots,b_p)_k=\prod_i(b_i)_k$. The two products are
\begin{equation}\label{eq:PK-products}
 \begin{aligned}
 W_k^{(0)}&=\frac{(-r,-\gamma/2-\ell)_k}{k!(1-\gamma/2)_k},
 \quad 0\le k\le r,\\[7pt]
 W_k^{(1)}&=\frac{(-\ell,\gamma/2-r)_k}{k!(1+\gamma/2)_k},
 \quad 0\le k\le\ell.
 \end{aligned}
\end{equation}
The finite Vandermonde identity
\[
 \sum_{k=0}^{R}\frac{(-R,b)_k}{k!(c)_k}
 =\frac{(c-b)_R}{(c)_R}
\]
gives their sums as
\begin{equation}\label{eq:PK-normalisations}
 \begin{aligned}
 Z_0&=\sum_{k=0}^{r}W_k^{(0)}
     =\frac{(\ell+1)_r}{(1-\gamma/2)_r},\\[7pt]
 Z_1&=\sum_{k=0}^{\ell}W_k^{(1)}
     =\frac{(r+1)_\ell}{(1+\gamma/2)_\ell}.
 \end{aligned}
\end{equation}
Consequently $\mathbf u^{(0)}=Z_0^{-1}\sum_{k=0}^r
W_k^{(0)}\delta_{2k}$ and
$\mathbf u^{(1)}=Z_1^{-1}\sum_{k=0}^{\ell}
W_k^{(1)}\delta_{2k+\gamma}$ are independent mass-one solutions.
\end{example}

The structural equation is linear in the functional.
If $\mathbf u^{(0)}$ and $\mathbf u^{(1)}$ solve it, then
\[
 \mathbf D\bigl(\phi(A\mathbf u^{(0)}+B\mathbf u^{(1)})\bigr)
 =\mathbf S\bigl(\psi(A\mathbf u^{(0)}+B\mathbf u^{(1)})\bigr),
 \quad A,B\in\mathbb C.
\]
In particular, suppose Proposition~\ref{lem:one-progression}
gives the two solutions
\begin{equation}\label{eq:two-components}
 \mathbf u^{(0)}=\sum_{k=0}^r\omega_k^{(0)}\delta_{X(a+k)},
 \quad
 \mathbf u^{(1)}=\sum_{k=0}^{\ell}\omega_k^{(1)}\delta_{X(c+k)},
\end{equation}
where all nodes in the two sums are distinct. Suppose also that
\[
 \langle\mathbf u^{(0)},1\rangle
 =\langle\mathbf u^{(1)},1\rangle=1.
\]
We say that these functionals have mass one. Their disjoint
supports make them linearly independent, and
\begin{equation}\label{eq:family-mixture}
 \mathbf u_\alpha=(1-\alpha)\mathbf u^{(0)}
                       +\alpha\mathbf u^{(1)},\quad \alpha\in\mathbb C,
\end{equation}
is a solution of \eqref{eq:structural} of mass one. Here $\alpha$
changes the relative masses while the calculus and structural pair
remain fixed. To decide whether these combinations exhaust the
normalised solutions, we use the triangular moment equations of
\cite[Section~7]{C26} and prove the required dimension count.

Put
\[
 \mu_n=\langle\mathbf u,x^n\rangle,\quad
 d_n=[x^{n+1}]\{\phi D(x^n)+\psi S(x^n)\},\quad n\ge0,
\]
where $[x^k]p$ means the coefficient of $x^k$ in $p$.
Since the monomials form a basis of $\mathcal P$, the structural
identity is equivalent to the equations
\[
 \langle\mathbf u,\phi D(x^n)+\psi S(x^n)\rangle=0
 ,\quad n\ge0.
\]
The degree properties in \eqref{eq:DS} give
$\deg\{\phi D(x^n)+\psi S(x^n)\}\le n+1$. Thus
\begin{equation}\label{eq:triangular-moments}
 d_n\mu_{n+1}+\sum_{k=0}^n a_{n,k}\mu_k=0,\quad n\ge0,
\end{equation}
where $a_{n,k}=[x^k]\{\phi D(x^n)+\psi S(x^n)\}$ for $0\le k\le n$.
All these coefficients are fixed by $X$ and $(\phi,\psi)$.
The system is triangular: when $d_n\ne0$, its $n$th equation
determines $\mu_{n+1}$ from the preceding moments. Throughout the
next proposition these equations are imposed for all $n\ge0$.
This global scope differs from the finite moment criterion of
\cite[Theorem~7.1]{C26}, which tests only a prescribed finite
polynomial space.

\begin{proposition}\label{prop:free-moment}
If $d_n\ne0$ for every $n\ge0$, the space of solutions of
\eqref{eq:structural} in $\mathcal P'$ has dimension one.
Fixing $\mu_0=1$ gives exactly one solution.

Suppose instead that, for some $m\ge0$, $d_m=0$ and
$d_n\ne0$ for $n\ne m$, and that a solution of mass one exists.
Then the solution space has dimension two. After fixing $\mu_0=1$,
the moments $\mu_1,\ldots,\mu_m$ are fixed, $\mu_{m+1}$ is free,
and all later moments are determined by it. Any two linearly
independent solutions of mass one therefore give every solution
of mass one through \eqref{eq:family-mixture}.
\end{proposition}
\begin{proof}
If every $d_n$ is non-zero, choose $\mu_0$ and solve
\eqref{eq:triangular-moments} in order. Each equation determines
one new moment. The sequence defines exactly one functional
on $\mathcal P$, since every polynomial is a finite linear
combination of monomials.

If only $d_m$ vanishes, the first $m$ equations determine
$\mu_1,\ldots,\mu_m$ as multiples of $\mu_0$. The equation
at $n=m$ then leaves a scalar multiple of $\mu_0$ equal to zero.
A solution with $\mu_0=1$ exists, so that scalar is zero.
The equation places no restriction on $\mu_0$ or $\mu_{m+1}$.
Every later equation again determines one new moment. These
two choices therefore determine exactly one functional, and
the solution space has dimension two.

Two independent solutions form a basis. If both have mass one,
their combination has mass one exactly when the coefficients
sum to one. This gives \eqref{eq:family-mixture}.
\end{proof}

\begin{example}[All para--Krawtchouk solutions]\label{ex:PK-moments}
Return to the pair in Example~\ref{ex:PK-masses}. For an affine
map, $D(x^n)$ has leading coefficient $n$ and $S(x^n)$ has
leading coefficient one. Writing $\phi_2=[x^2]\phi$ and
$\psi_1=[x]\psi$, we have $d_n=\phi_2n+\psi_1$, as in
\cite[Theorem~7.8]{C26}. This leading-coefficient identity does
not require every $d_n$ to be non-zero. Here $\phi_2=-1/m$ and $\psi_1=1$, so
\begin{equation}
 d_n=-\frac nm+1,\quad n\ge0,\quad
 d_{N-1}=0,\quad d_n\ne0,\quad n\ne N-1.
\end{equation}
The two finite sums already constructed show that a mass-one
solution exists, so the moment equation at $n=N-1$ is compatible.
Proposition~\ref{prop:free-moment} now proves that these sums give
all solutions. Every normalised solution is
\[
 \mathbf u_\alpha
 =\frac{1-\alpha}{Z_0}\sum_{k=0}^rW_k^{(0)}\delta_{2k}
  +\frac{\alpha}{Z_1}\sum_{k=0}^{\ell}W_k^{(1)}\delta_{2k+\gamma}.
\]
In particular, the two components have the same moments through
degree $N-1$. Their moments of degree $N$ differ; otherwise the
remaining triangular equations would make the functionals equal.
The free moment is therefore $\mu_N$, and changing $\alpha$
changes it. No new map or structural pair is needed.

For $0<\alpha<1$, all $N+1$ weights are positive, so the
functional is regular of order $N+1$. At $\alpha=0,1$ only
one component remains. We shall now determine the recurrence
for the full combination.
\end{example}

\begin{remark}
The conditions $d_m=0$ and $d_n\ne0$ for $n\ne m$ are hypotheses of
Proposition~\ref{prop:free-moment}; the degree bounds on
$\phi,\psi$ alone do not impose them. For the four specific pairs
in Table~\ref{tab:structural-location},
\eqref{eq:four-dn-lambda} verifies that the only vanishing
coefficient is $d_{N-1}$, assuming that $q$ is not a root of
unity in the $q$-cases. Once a mass-one solution exists, the
proposition gives a two-dimensional solution space and a second
independent mass-one solution. Finite support does not follow
from this dimension count; in the applications it is proved by
the endpoint construction. The alternating calculus can leave several moments free
and need not satisfy these hypotheses.
\end{remark}

\subsection{Recovering the finite recurrence}

We now pass from global structural solutions to finite
orthogonality. The following argument applies to any two finite
functionals with the stated hypotheses, whether or not they satisfy
a structural equation. Let the two
normalised functionals in \eqref{eq:two-components} be regular
of orders $r+1$ and $\ell+1$. This means
$\Delta_n^{(0)}\ne0$ for $0\le n\le r$ and
$\Delta_n^{(1)}\ne0$ for $0\le n\le\ell$, where $\Delta_n^{(i)}$
is the Hankel determinant of the moments of $\mathbf u^{(i)}$.
Their supports $E_0,E_1$ contain
$N+1=r+\ell+2$ distinct nodes altogether. Assume that their
moments agree through degree $N-1$, and that their sizes are
balanced as in \eqref{eq:r-ell}. Put $j=\lfloor N/2\rfloor$ and
take $\alpha\in\mathbb C\setminus\{0,1\}$. The endpoint
construction and Proposition~\ref{prop:free-moment} provide such common
moments in the four applications, but are not assumptions of the
recurrence result itself.

Write $Q_n^{(i)}$ for the component monic polynomials, including
their terminal polynomials, and $B_n^{(i)},C_n^{(i)}$ for their
recurrence coefficients. The coefficients of a monic orthogonal
polynomial of degree $n$ are determined by moments through degree
$2n-1$. The component regularity ensures that the determining
system is nonsingular, so
\[
 Q_n^{(0)}=Q_n^{(1)}=:Q_n,\quad 0\le n\le j.
\]
The norms use moments through degree $2n$, and the diagonal
coefficients use moments through degree $2n+1$. Thus the
common recurrence coefficients have the following ranges:
\begin{subequations}\label{eq:common-recurrence-ranges}
\begin{equation}
 \begin{aligned}
 B_n^{(0)}=B_n^{(1)}&=:B_n,\quad 0\le n<j,\\[7pt]
 C_n^{(0)}=C_n^{(1)}&=:C_n,\quad 1\le n\le j,
 \end{aligned}
 \quad N=2j+1;
\end{equation}
\begin{equation}
 \begin{aligned}
 B_n^{(0)}=B_n^{(1)}&=:B_n,\quad 0\le n<j,\\[7pt]
 C_n^{(0)}=C_n^{(1)}&=:C_n,\quad 1\le n<j,
 \end{aligned}
 \quad N=2j.
\end{equation}
\end{subequations}
In the even case, $Q_j$ is the smaller component's terminal
polynomial. Let
\[
 \Pi_i(x)=\prod_{z\in E_i}(x-z),\quad
 h_k=\langle\mathbf u^{(0)},Q_k^2\rangle
     =\langle\mathbf u^{(1)},Q_k^2\rangle\ne0,
 \quad 0\le k\le\ell.
\]
We shall construct the polynomials $P_n$ of the combination
$\mathbf u_\alpha$ and their squared norms
$H_n=\langle\mathbf u_\alpha,P_n^2\rangle$.

For $N=2j+1$, the terminal component recurrences are
\[
 \Pi_i=(x-b^{(i)})Q_j-C_jQ_{j-1},\quad
 b^{(i)}=B_j^{(i)},\quad i=0,1.
\]
Orthogonality to $Q_j$ fixes
$B_j=(1-\alpha)b^{(0)}+\alpha b^{(1)}$, hence
\begin{equation}\label{eq:construct-first-odd}
 P_{j+1}=(1-\alpha)\Pi_0+\alpha\Pi_1.
\end{equation}
Writing $K=b^{(1)}-b^{(0)}$, subtraction gives
$\Pi_1-\Pi_0=-KQ_j$, with $K\ne0$ because the supports are
disjoint. Thus $P_{j+1}$ equals $-\alpha KQ_j$ on $E_0$ and
$(1-\alpha)KQ_j$ on $E_1$.

For either parity, let $R_{N-k}$ be the unique interpolating
polynomial of degree at most $N$ with
\begin{equation}\label{eq:interpolated-polynomials}
 R_{N-k}(x)=
 \begin{cases}
 -\alpha Q_k(x),&\text{if }x\in E_0,\\[7pt]
 (1-\alpha)Q_k(x),&\text{if }x\in E_1,
 \end{cases}
 \quad 0\le k\le\ell.
\end{equation}
If $\deg p<N-k$, then $\deg(Q_kp)\le N-1$. The common
moments therefore give
\[
 \langle\mathbf u_\alpha,R_{N-k}p\rangle
 =\alpha(1-\alpha)
  \langle\mathbf u^{(1)}-\mathbf u^{(0)},Q_kp\rangle=0.
\]
The next lemma proves the indicated degrees and regularity;
neither is assumed for the combination.

\begin{lemma}\label{lem:combined-polynomials}
Under the assumptions above, $\mathbf u_\alpha$ is regular
of order $N+1$ for every $\alpha\ne0,1$. Its monic orthogonal
polynomials satisfy
\[
 P_n=Q_n,\quad 0\le n\le j.
\]
Put
\[
 M_N=\langle\mathbf u^{(1)}-\mathbf u^{(0)},x^N\rangle.
\]
Then $M_N\ne0$, $\deg R_{N-k}=N-k$, and
\begin{equation}\label{eq:construct-remaining}
 P_{N-k}=\frac{M_N}{h_k}R_{N-k},
 \quad
 H_{N-k}=\alpha(1-\alpha)\frac{M_N^2}{h_k},
 \quad 0\le k\le\ell.
\end{equation}
The terminal polynomial is $P_{N+1}=\Pi_0\Pi_1$.
\end{lemma}

\begin{proof}
The displayed orthogonality and the prescribed values give
\[
 \langle\mathbf u_\alpha,R_{N-k}R_{N-t}\rangle
 =\alpha(1-\alpha)h_k\delta_{k,t},
 \quad 0\le k,t\le\ell.
\]
Starting at $R_N$, its non-zero norm and orthogonality to all
lower degrees force degree $N$: a lower degree would make it
orthogonal to itself. Suppose the degrees of
$R_{d+1},\ldots,R_N$ have been established. Subtracting their
leading terms writes $R_d=\sum_{n>d}c_nR_n+T$, with
$\deg T\le d$. Each $R_n$ is orthogonal to $T$ and to $R_d$.
Its non-zero norm therefore forces $c_n=0$. The norm of $R_d$
then forces
$\deg R_d=d$. This proves all the claimed degrees.

Since $Q_kx^{N-k}$ is monic of degree $N$,
\[
 \langle\mathbf u_\alpha,R_{N-k}x^{N-k}\rangle
 =\alpha(1-\alpha)M_N.
\]
If $a_{N-k}$ is the leading coefficient, comparison with its
squared norm gives $h_k=a_{N-k}M_N$. The terms below degree
$N-k$ disappear by orthogonality. Since $h_k\ne0$, we have
$M_N\ne0$, and
\eqref{eq:construct-remaining} follows. The remaining lower
norms are $H_k=h_k$ for $k\le\ell$; when $N=2j$,
\[
 H_j=(1-\alpha)\langle\mathbf u^{(0)},Q_j^2\rangle\ne0.
\]
The monic change of basis has determinant one, so
$\Delta_n=\prod_{k=0}^{n}H_k\ne0$ for $0\le n\le N$.
Thus the functional is regular of order $N+1$, and
$\Pi_0\Pi_1$, which vanishes on the entire
support, is its terminal polynomial.
\end{proof}

Formula \eqref{eq:construct-remaining} gives the remaining
polynomials and their norms. We now use it to compute the
central recurrence coefficients and obtain the reflected coefficients
in Proposition~\ref{prop:central-block}.

\begin{proposition}\label{prop:central-block}
Let the two components in \eqref{eq:two-components} have mass
one, be regular of orders $r+1,\ell+1$, and have the same moments
through degree $N-1$, where $N=r+\ell+1$. Assume their $N+1$
nodes are distinct and $r,\ell$ are as in \eqref{eq:r-ell}.
For $\alpha\in\mathbb C\setminus\{0,1\}$, the functional
$\mathbf u_\alpha$ is regular of order $N+1$. Its recurrence
coefficients are given as follows.

If $N=2j+1$, retain the common coefficients
$B_0,\ldots,B_{j-1}$ and $C_1,\ldots,C_j$. Put
$b^{(i)}=B_j^{(i)}$ for $i=0,1$. The central coefficients are
\[
 \begin{aligned}
 B_j&=(1-\alpha)b^{(0)}+\alpha b^{(1)},\\[7pt]
 B_{j+1}&=\alpha b^{(0)}+(1-\alpha)b^{(1)},\\[7pt]
 C_{j+1}&=\alpha(1-\alpha)(b^{(1)}-b^{(0)})^2.
 \end{aligned}
\]
The remaining entries satisfy
\[
 \begin{aligned}
 B_{N-n}&=B_n,\quad 0\le n<j,\\[7pt]
 C_{N+1-n}&=C_n,\quad 1\le n\le j.
 \end{aligned}
\]

If $N=2j$, retain $B_0,\ldots,B_{j-1}$ and
$C_1,\ldots,C_{j-1}$. Put $b^{(0)}=B_j^{(0)}$ and
$c^{(0)}=C_j^{(0)}$. The central coefficients are
\[
 B_j=b^{(0)},\quad
 C_j=(1-\alpha)c^{(0)},\quad
 C_{j+1}=\alpha c^{(0)}.
\]
The remaining entries satisfy
\[
 \begin{aligned}
 B_{N-n}&=B_n,\quad 0\le n<j,\\[7pt]
 C_{N+1-n}&=C_n,\quad 1\le n<j.
 \end{aligned}
\]
\end{proposition}

\begin{proof}
Lemma~\ref{lem:combined-polynomials} proves regularity and
identifies the polynomials and their norms. The first recurrence
coefficients are already fixed by
\eqref{eq:common-recurrence-ranges}. To compute the remaining
ones, orthogonality and \eqref{eq:TTRR} give
\[
 \begin{aligned}
 B_n&=\frac{\langle\mathbf u_\alpha,xP_n^2\rangle}{H_n},
 \quad 0\le n\le N,\\[7pt]
 C_n&=\frac{H_n}{H_{n-1}},\quad 1\le n\le N.
 \end{aligned}
\]

First suppose $N=2j+1$. We have $H_j=h_j$, and the construction
leading to \eqref{eq:construct-first-odd} already gives
\[
 B_j=(1-\alpha)b^{(0)}+\alpha b^{(1)}.
\]
Put $K=b^{(1)}-b^{(0)}$. Since $xQ_j^2$ is monic of degree
$N$ and the lower moments of the two components agree,
\[
 M_N
 =\langle\mathbf u^{(1)}-\mathbf u^{(0)},xQ_j^2\rangle
 =(b^{(1)}-b^{(0)})h_j
 =Kh_j.
\]
Formula \eqref{eq:construct-remaining}, with $k=j$, therefore
gives
\[
 H_{j+1}=\alpha(1-\alpha)K^2h_j,
 \quad
 C_{j+1}=\frac{H_{j+1}}{H_j}
        =\alpha(1-\alpha)K^2.
\]
The values of $P_{j+1}$ on the two supports give
\begin{align*}
 \langle\mathbf u_\alpha,xP_{j+1}^2\rangle
 ={}&(1-\alpha)\alpha^2K^2
       \langle\mathbf u^{(0)},xQ_j^2\rangle\\[7pt]
 &+\alpha(1-\alpha)^2K^2
       \langle\mathbf u^{(1)},xQ_j^2\rangle\\[7pt]
 ={}&\alpha(1-\alpha)K^2h_j
       \bigl(\alpha b^{(0)}+(1-\alpha)b^{(1)}\bigr).
\end{align*}
Dividing by $H_{j+1}$ gives the stated value of $B_{j+1}$.

Now suppose $N=2j$. Here $P_j=Q_j=\Pi_1$, so $Q_j$
vanishes on the smaller support. The last row of the larger
component is
\[
 \Pi_0=(x-b^{(0)})Q_j-c^{(0)}Q_{j-1}.
\]
Its last non-zero norm is
\[
 \langle\mathbf u^{(0)},Q_j^2\rangle
 =c^{(0)}h_{j-1}.
\]
Consequently,
\[
 H_j=(1-\alpha)c^{(0)}h_{j-1},
 \quad
 \langle\mathbf u_\alpha,xQ_j^2\rangle=b^{(0)}H_j.
\]
Since $H_{j-1}=h_{j-1}$, these equalities give
\[
 B_j=b^{(0)},\quad C_j=(1-\alpha)c^{(0)}.
\]
Also, $Q_j^2$ is monic of degree $N$, and hence
\[
 M_N
 =\langle\mathbf u^{(1)}-\mathbf u^{(0)},Q_j^2\rangle
 =-c^{(0)}h_{j-1}.
\]
Using \eqref{eq:construct-remaining} with $k=j-1$, we obtain
\[
 H_{j+1}
 =\alpha(1-\alpha)(c^{(0)})^2h_{j-1},
 \quad
 C_{j+1}=\frac{H_{j+1}}{H_j}=\alpha c^{(0)}.
\]

It remains to recover the reflected entries. For either parity
and $0\le k<j$, the common component coefficient gives
\[
 \langle\mathbf u^{(i)},xQ_k^2\rangle=B_kh_k,
 \quad i=0,1.
\]
By \eqref{eq:interpolated-polynomials} and
\eqref{eq:construct-remaining},
\begin{align*}
 \langle\mathbf u_\alpha,xP_{N-k}^2\rangle
 ={}&\frac{M_N^2}{h_k^2}
 \bigl((1-\alpha)\alpha^2
       \langle\mathbf u^{(0)},xQ_k^2\rangle
       +\alpha(1-\alpha)^2
       \langle\mathbf u^{(1)},xQ_k^2\rangle\bigr)\\[7pt]
 ={}&\alpha(1-\alpha)\frac{M_N^2}{h_k}B_k
 =H_{N-k}B_k.
\end{align*}
Thus $B_{N-k}=B_k$.

For $0\le k<\ell$, both consecutive norms are given by
\eqref{eq:construct-remaining}, so
\[
 C_{N-k}
 =\frac{H_{N-k}}{H_{N-k-1}}
 =\frac{h_{k+1}}{h_k}
 =C_{k+1}.
\]
These are exactly the claimed reflections. The terminal
polynomial was already identified in
Lemma~\ref{lem:combined-polynomials}.
\end{proof}

The last component entries can also be computed directly
from their nodes. Put
\[
 s_i=\sum_{x\in E_i}x,\quad
 t_i=\sum_{x\in E_i}x^2,\quad i=0,1.
\]
For odd $N$, the coefficient of $x^{j-1}$ in $Q_j$ is
$-\sum_{n=0}^{j-1}B_n$. Comparing the next-to-leading
coefficients in the terminal recurrences gives
\begin{equation}\label{eq:central-node-sums-odd}
 b^{(0)}=s_0-\sum_{n=0}^{j-1}B_n,\quad
 b^{(1)}-b^{(0)}=s_1-s_0.
\end{equation}

For even $N$, use
\[
 \Pi_0=(x-b^{(0)})\Pi_1-c^{(0)}Q_{j-1}.
\]
Comparing the coefficients of $x^j$ gives
$b^{(0)}=s_0-s_1$. The sums of pairwise node products are
$(s_i^2-t_i)/2$, with an empty sum equal to zero. Comparing
at degree $j-1$ gives
\[
 \frac{s_0^2-t_0}{2}
 =\frac{s_1^2-t_1}{2}+b^{(0)}s_1-c^{(0)}.
\]
Therefore
\begin{equation}\label{eq:central-node-sums-even}
 b^{(0)}=s_0-s_1,\quad
 c^{(0)}=\frac{t_0-t_1-(b^{(0)})^2}{2}.
\end{equation}

This completes the recurrence construction. The real positive setting,
including the common moments of the two restrictions and central
isospectral deformations, is treated in \cite[Sections~3--4]{GTVZ17}.
Here the component argument works over $\mathbb C$ and proves regularity
for every $\alpha\ne0,1$; Proposition~\ref{prop:free-moment} supplies
the exhaustiveness of the global structural solutions in the four
applications. Only the central coefficients depend on $\alpha$;
the remaining entries follow by reflection. At $\alpha=1/2$ the
recurrence is persymmetric. Positivity requires additional real
parameter conditions.

For the four identifications below, we compute the common recurrence
coefficients from \cite[Theorems~7.5 and~7.8]{C26}, in the ranges
where their denominators and the component Hankel determinants are
non-zero. More precisely, their finite version for degrees
$0,\ldots,M+1$ assumes $d_0,\ldots,d_{2M+1}\ne0$.
For $N=2j+1$, take $M=j-1$, which gives all the common recurrence coefficients in
\eqref{eq:common-recurrence-ranges}. For $N=2j$ and $j\ge2$,
take $M=j-2$; the last common diagonal entry $B_{j-1}$ is
obtained from the same local construction, whose divisors still
precede $d_{2j-1}=0$; see \cite[pp.~39 and~51--52]{C26}.
For $j=1$, $B_0$ follows directly by testing
\eqref{eq:structural} on $1$. Proposition~\ref{prop:central-block}
gives the remaining recurrence coefficients directly from the
component functionals, without applying those formulas at the
excluded indices.

\section{Identification of the para-families}\label{sec:families}
We now verify Table~\ref{tab:structural-location}, with $r,\ell,m$
as in \eqref{eq:r-ell}. Throughout, assume the displayed
expressions are defined, the nodes are distinct, the component
weights and normalising sums are non-zero, and both components
are regular of their finite orders. Take $\alpha\ne0,1$.
For complex parameters, component regularity is a hypothesis;
distinct nodes and non-zero weights do not imply it. In the $q$-cases,
take $q\ne0$ not a root of unity, with the logarithm fixed above.
Endpoint zeros and ratios refer to individually defined $\lambda_\pm$;
rational continuation applies to the complete structural identity.

We identify the full published functionals by their normalised
weights, denoted by $w_k^{\rm pub}$, using the map and structural
pair in the table. Equality of the functionals gives equality of
all monic orthogonal polynomials and recurrence entries, including
the terminal row. Their common entries can be computed directly
from \cite[Theorems~7.5 and~7.8]{C26} in the ranges justified
above; Proposition~\ref{prop:central-block} supplies the remaining coefficients
and their reflection. We retain the full para--Krawtchouk recurrence as
an illustration and refer to the original papers for the other
published lists. In each comparison we retain the parameters and polynomial variable
of the cited source.
Published para--Krawtchouk uses $\alpha=1/2$; the other three
allow a free $\alpha$. For \cite{VZ12,LVZ16,LVZ18}, formula references
use the arXiv versions specified in the bibliography.

For the $q$-products, write
\[
 \left(a;q\right)_k
 =\prod_{i=0}^{k-1}\left(1-aq^i\right),\quad
 \left(a_1,\ldots,a_p;q\right)_k
 =\prod_{i=1}^{p}\left(a_i;q\right)_k.
\]
Empty products are one. We normalise each component by writing
$\omega_k=W_k/Z$, where $W_0=1$ and $Z=\sum_kW_k$,
then multiply its weights by $1-\alpha$ or $\alpha$.

The standard formulas in \cite[Theorems~7.5 and~7.8]{C26} give
\begin{subequations}\label{eq:four-dn-lambda}
\begin{equation}
 d_n=1-\frac{n}{m}
\end{equation}
for para--Krawtchouk and para--Racah, and
\begin{equation}
 d_n=\frac{q^{-m}t^n-t^{-n}}{q^{-m}-1}
\end{equation}
for $q$-para--Racah and $q$-para--Krawtchouk.
\end{subequations}
In each case, $d_{N-1}=0$ is the only zero. Once the two
components are constructed, Proposition~\ref{prop:free-moment}
gives all normalised solutions of the structural equation.
Proposition~\ref{prop:central-block} then gives their full
recurrence and regularity for \mbox{$\alpha\ne0,1$}.

\subsection{Para--Krawtchouk}
The published nodes and normalised weights
\cite[Eqs.~(3.4)--(3.5) and Section~6]{VZ12} are
\[
 x_{2k}=2k,\quad 0\le k\le r,\quad
 x_{2k+1}=2k+\gamma,\quad 0\le k\le\ell,
\]
and
\begin{equation}\label{eq:PK-published-weights}
 \begin{aligned}
 w_{2k}^{\rm pub}
 &=\frac{\left(1-\gamma/2\right)_r}
         {2\left(\ell+1\right)_r}
   \frac{\left(-r,-\gamma/2-\ell\right)_k}
        {k!\left(1-\gamma/2\right)_k},\\[7pt]
 w_{2k+1}^{\rm pub}
 &=\frac{\left(1+\gamma/2\right)_\ell}
         {2\left(r+1\right)_\ell}
   \frac{\left(-\ell,\gamma/2-r\right)_k}
        {k!\left(1+\gamma/2\right)_k}.
 \end{aligned}
\end{equation}
For odd $N$, the monic recurrence coefficients are
\begin{equation}\label{eq:PK-recurrence}
 B_n^{\rm pub}=\frac{N-1+\gamma}{2},\quad
 C_n^{\rm pub}
 =\frac{n\left(N+1-n\right)
        \left(\left(N+1-2n\right)^2-\gamma^2\right)}
        {4\left(N-2n\right)\left(N-2n+2\right)}.
\end{equation}
For even $N$, put $M=N+1$. The weights
\eqref{eq:PK-published-weights} give the monic recurrence
\begin{equation}\label{eq:PK-even-recurrence}
 \begin{aligned}
 B_n^{\rm pub}
 &=\frac{M+\gamma}{2}-1+
   \frac{\left(\gamma-1\right)M}{4}
   \left(\frac1{2n-M}-\frac1{2n+2-M}\right),\\[7pt]
 C_n^{\rm pub}
 &=\frac{n\left(M-n\right)
        \left(\left(2n-M\right)^2-\left(\gamma-1\right)^2\right)}
        {4\left(2n-M\right)^2}.
 \end{aligned}
\end{equation}
These formulas use monic normalisation, with $0\le n\le N$ for
$B_n^{\rm pub}$ and $1\le n\le N$ for $C_n^{\rm pub}$.

Examples~\ref{ex:PK-masses} and~\ref{ex:PK-moments}, with
$\alpha=1/2$, give the weights \eqref{eq:PK-published-weights}.
Equality of these finite functionals identifies every monic polynomial
and recurrence entry, including the terminal row, and proves classicality
for the stated map and structural pair.

\subsection{Para--Racah}
\label{subsec:PR}
The published support \cite{LVZ16} is
\[
 x_{2k}=-\left(a+k\right)^2
       ,\quad 0\le k\le r,\quad
 x_{2k+1}=-\left(c+k\right)^2
       ,\quad 0\le k\le\ell.
\]
To write both components without repeating the same expression,
put
\begin{equation}\label{eq:R-weights-full}
 W_k^{(0)}
 =\frac{a+k}{a}
  \frac{\left(2a,-r,a+c,a-c-\ell\right)_k}
       {k!\left(a-c+1,2a+r+1,a+c+\ell+1\right)_k},
\end{equation}
and
\begin{equation}\label{eq:R-normalisation}
 Z_0=\sum_{k=0}^{r}W_k^{(0)}
 =\frac{\left(2a+1,\ell+1\right)_r}
        {\left(a-c+1,a+c+\ell+1\right)_r}.
\end{equation}
Define $W_k^{(1)},Z_1$ by interchanging
$\left(a,r\right)$ and $\left(c,\ell\right)$.
The published weights are then
\[
 w_{2k}^{\rm pub}
 =(1-\alpha)\frac{W_k^{(0)}}{Z_0},\quad
 w_{2k+1}^{\rm pub}
 =\alpha\frac{W_k^{(1)}}{Z_1},
\]
as in \cite[Eqs.~(3.14), (3.17), and (4.13)]{LVZ16}.

\paragraph{\textbf{Checking the weights.}}
Take $X(s)=-s^2$ and the para--Racah pair in
Table~\ref{tab:structural-location}. Since $Y(s)-Z(s)=-2s$,
substitution in \eqref{eq:lambdas} gives
\[
 \lambda_-(s)
 =\frac{\left(s-a\right)\left(s-c\right)
         \left(s+a+r\right)\left(s+c+\ell\right)}{2ms},
 \quad
 \lambda_+(s)=-\lambda_-(-s).
\]
The lower endpoint zeros are $a,c$, and the upper endpoint
zeros are $a+r,c+\ell$. Moreover,
\[
 \frac{\lambda_+\left(a+k\right)}
      {\lambda_-\left(a+k+1\right)}
 =\frac{W_{k+1}^{(0)}}{W_k^{(0)}}.
\]
Thus Proposition~\ref{lem:one-progression} gives the weights
\eqref{eq:R-weights-full}. Their sum
\eqref{eq:R-normalisation} follows from the standard terminating
${}_5F_4$ identity \cite[Eq.~(1.5.6), terminating specialization]{KLS10}.
Interchanging the parameters gives the
second component. Combining them with coefficients $1-\alpha$
and $\alpha$ therefore recovers the published functional.

\subsection{\texorpdfstring{$q$-para--Racah}{q-para-Racah}}

The published support \cite{LVZ18} is
\[
 x_{2k}=\chi\left(aq^k\right)
       ,\quad 0\le k\le r,\quad
 x_{2k+1}=\chi\left(cq^k\right)
       ,\quad 0\le k\le\ell,
 \quad
 \chi(z)=\frac{z+z^{-1}}2.
\]
As before, write the two components using one product:
\begin{equation}\label{eq:qR-weights-full}
 W_k^{(0)}
 =q^{\left(m+1\right)k}\frac{1-a^2q^{2k}}{1-a^2}
  \frac{\left(a^2,ac,q^{-r},aq^{-\ell}/c;q\right)_k}
       {\left(q,aq/c,a^2q^{r+1},acq^{\ell+1};q\right)_k}.
\end{equation}
Its normalisation is
\begin{equation}\label{eq:qR-normalisation}
 Z_0=\sum_{k=0}^{r}W_k^{(0)}
 =\frac{\left(a^2q,q^{\ell+1};q\right)_r}
        {\left(aq/c,acq^{\ell+1};q\right)_r}.
\end{equation}
Define $W_k^{(1)},Z_1$ by interchanging
$\left(a,r\right)$ and $\left(c,\ell\right)$.
Dividing the published weights by their initial values gives
$W_k^{(0)}$ and $W_k^{(1)}$
\cite[Eqs.~(3.33)--(3.34) and~(4.19)--(4.20)]{LVZ18}.
Their component masses \cite[Eqs.~(3.37) and~(4.22)]{LVZ18}
fix the normalisation:
\[
 w_{2k}^{\rm pub}
 =(1-\alpha)\frac{W_k^{(0)}}{Z_0},\quad
 w_{2k+1}^{\rm pub}
 =\alpha\frac{W_k^{(1)}}{Z_1}.
\]

Put $b=a^{-1}q^{-r}$, $d=c^{-1}q^{-\ell}$ and $E_4=abcd=q^{-m}$.

\paragraph{\textbf{Checking the weights.}}
Take $X(s)=\chi\left(q^s\right)$ and the $q$-para--Racah
pair in Table~\ref{tab:structural-location}. With $z=q^s$,
substitution in \eqref{eq:lambdas} gives
\[
 \lambda_-(s)
 =\frac{\left(z-a\right)\left(z-b\right)
         \left(z-c\right)\left(z-d\right)}
        {2\left(E_4-1\right)z\left(z^2-1\right)},
 \quad
 \lambda_+(s)=-\lambda_-(-s).
\]
The lower endpoint zeros are $z=a,c$, and the upper endpoint
zeros are $z=aq^r,cq^\ell$. On the first progression,
$q^s=aq^k$, direct substitution gives
\[
 \frac{\lambda_+(s)}{\lambda_-\left(s+1\right)}
 =\frac{W_{k+1}^{(0)}}{W_k^{(0)}}.
\]
Thus Proposition~\ref{lem:one-progression} gives
\eqref{eq:qR-weights-full}. The sum
\eqref{eq:qR-normalisation} follows from the standard
terminating ${}_6\phi_5$ identity \cite[Eq.~(1.11.11)]{KLS10}. Interchanging the parameters
gives the second component. Their linear combination therefore
recovers the normalised functional.

\subsection{\texorpdfstring{$q$-para--Krawtchouk}{q-para-Krawtchouk}}

These polynomials arise from $q$-para--Racah by keeping
$\Delta=a/c$ fixed, rescaling its variable by $2/c$, and letting
$c\to\infty$ \cite[Section~5.1]{LVZ18}. Their support is
\[
 x_{2k}=\Delta q^k,\quad 0\le k\le r,\quad
 x_{2k+1}=q^k,\quad 0\le k\le\ell.
\]
Put
\begin{equation}\label{eq:qK-weights-full}
 W_k^{(0)}
 =q^k\frac{\left(q^{-r},\Delta q^{-\ell};q\right)_k}
              {\left(q,\Delta q;q\right)_k},
\end{equation}
and
\begin{equation}\label{eq:qK-normalisation}
 Z_0=\sum_{k=0}^{r}W_k^{(0)}
 =\left(\Delta q^{-\ell}\right)^r
   \frac{\left(q^{\ell+1};q\right)_r}
        {\left(\Delta q;q\right)_r}.
\end{equation}
Define $W_k^{(1)},Z_1$ by replacing
$\left(\Delta,r,\ell\right)$ with
$\left(\Delta^{-1},\ell,r\right)$.
The published weights \cite[Eqs.~(5.14)--(5.18)]{LVZ18} are
\[
 w_{2k}^{\rm pub}
 =\left(1-\alpha\right)\frac{W_k^{(0)}}{Z_0},\quad
 w_{2k+1}^{\rm pub}
 =\alpha\frac{W_k^{(1)}}{Z_1}.
\]

\paragraph{\textbf{Checking the weights.}}
Take $X(s)=q^s$ and the $q$-para--Krawtchouk pair in
Table~\ref{tab:structural-location}. With $z=q^s$,
\eqref{eq:lambdas} gives
\[
 \lambda_-(s)
 =\frac{\left(z-\Delta\right)\left(z-1\right)}
        {\left(q^{-m}-1\right)z},\quad
 \lambda_+(s)
 =\frac{q^{-m}\left(z-\Delta q^r\right)\left(z-q^\ell\right)}
        {\left(q^{-m}-1\right)z}.
\]
These vanish at the required starts and ends. On the first
progression, $q^s=\Delta q^k$, substitution gives
\[
 \frac{\lambda_+(s)}{\lambda_-\left(s+1\right)}
 =\frac{W_{k+1}^{(0)}}{W_k^{(0)}}.
\]
Proposition~\ref{lem:one-progression} therefore gives
\eqref{eq:qK-weights-full}. The sum
\eqref{eq:qK-normalisation} follows from the standard
terminating ${}_2\phi_1$ identity \cite[Eq.~(1.11.5)]{KLS10}.
The parameter replacement
above gives the second component. Their linear combination
recovers the published weights and normalisation.

\subsection{Para--Bannai--Ito}\label{sec:para-BI}

The complementary family follows the alternating reconstruction in
\cite[Proposition~1.3]{CG26b}; a Geronimus inversion gives its
Bannai partner. We prove the finite form below, including the
terminal row. For $X(s)=e^{\pi is}$,
the structural equation is equivalent to
\begin{equation}\label{eq:alternating-annihilation}
 \langle\mathbf u,\phi(x)p(x^2)\rangle=0,\quad
 \langle\mathbf u,\psi(x)p(x^2)\rangle=0
 ,\quad p\in\mathcal P,
\end{equation}
by \cite[Proposition~8.3]{C26}.

Fix $j\ge1$ and use the parameters $a,b,\alpha$ of \cite{PVZ23}.
Write $\mathbf u^{\rm C},\mathbf u^{\rm B}$ for the complementary
and Bannai functionals in \cite{PVZ23}, with truncation indices
$2j$ and $2j+1$, respectively. Both functionals are normalised to mass one. We assume
both systems are regular and their nodes are distinct. Write
$P_n^{\rm C},P_n^{\rm B}$ for their monic polynomials, and use
the same superscripts for their recurrence coefficients. Put
\begin{equation}\label{eq:BI-parameters}
 \begin{gathered}
 \rho=\frac{b-j-1+a}{4},\\[7pt]
 A=\frac{b-j-1-a}{4},\quad\text{if }j\text{ is even},\quad
 A=-\frac{b+j+1+a}{4},\quad\text{if }j\text{ is odd}.
 \end{gathered}
\end{equation}

The published recurrences give
\[
 B_n^{\rm C}=(-1)^nA,\quad 0\le n\le2j.
\]
Their Christoffel--Geronimus connection \cite{CG26b}, given in
\cite[Eqs.~(2.4)--(2.5), Section~4.3 and Appendix~B]{PVZ23}, gives
\begin{equation}\label{eq:BI-Geronimus}
 \left(x-\rho\right)\mathbf u^{\rm B}
 =\kappa\mathbf u^{\rm C},\quad
 \kappa=B_0^{\rm B}-\rho\ne0.
\end{equation}
Testing on $1$ gives $\kappa$. Write $F_n^{\rm B},G_n^{\rm B}$
for the auxiliary Bannai recurrence factors denoted $A_n,C_n$ in
\cite{PVZ23}. They satisfy $\kappa=-F_0^{\rm B}$, and the first
monic subdiagonal coefficient is $F_0^{\rm B}G_1^{\rm B}\ne0$.
Thus regularity ensures $\kappa\ne0$.

\begin{proposition}\label{prop:BI-classical}
The complementary functional is classical in the alternating
calculus with $\left(\phi,\psi\right)=\left(0,x-A\right)$.
In addition, its Bannai partner is classical in the same calculus with
$\left(\phi,\psi\right)
 =\left(\left(x-\rho\right)\left(x-A\right),0\right)$.
\end{proposition}

\begin{proof}
The terminal row is included in the published alternating diagonal.
Let $J$ be its finite monic recurrence matrix, with diagonal
$B_n^{\rm C}$, lower diagonal one and upper diagonal $C_{n+1}^{\rm C}$,
and put $H=\diag(1,-1,1,\ldots,1)$. Then $J=AH+K$, where $K$ has
zero diagonal and $HK=-KH$. Consequently
\[
 J^2=A^2I+K^2,\quad HJ^2=J^2H.
\]
The finite recurrence and its terminal condition represent the
normalised functional as
$\langle\mathbf u^{\rm C},p\rangle=e_0^{\mathsf T}p(J)e_0$,
where $e_0$ is the first coordinate vector and $\mathsf T$ denotes transpose.
Indeed, both functionals agree on the orthogonal basis through
degree $2j$ and are annihilated by its terminal polynomial;
polynomial division identifies them on all of $\mathcal P$.
Since $He_0=e_0$ and $K$ exchanges the two parity subspaces,
$e_0^{\mathsf T}K(J^2)^ne_0=0$ for every $n\ge0$. It follows that
\[
 \langle\mathbf u^{\rm C},(x-A)x^{2n}\rangle=0,
 \quad n\ge0.
\]
This proves the complementary annihilation identity on all polynomials,
including degrees above the finite orthogonal system. For the Bannai
functional, \eqref{eq:BI-Geronimus} gives:
\[
 \langle\mathbf u^{\rm B},
 (x-\rho)(x-A)p(x^2)\rangle=0.
\]
Criterion \eqref{eq:alternating-annihilation} proves both claims.
\end{proof}

\paragraph{\textbf{Recovering the weights.}}
The explicit product calculation below assumes non-zero, distinct
lifts and defined weight expressions. The matrix proof above also
covers regular distinct-node systems with a zero lift. An unpaired
quadratic mass $\nu_0$ at $t=A=0$ lifts to $\nu_0\delta_0$;
no division by $t$ is used for that contribution.
For even $j$, take the para--Racah system of truncation index $j$
with $(A,C,\omega)$ in place of its parameters $(a,c,\alpha)$, respectively, where
\[
 A=\frac{b-j-1-a}{4},\quad
 C=\frac{b-j+3+a}{4},\quad \omega=1-\alpha.
\] Substitution
in the published recurrence identifies the quadratic component
as
\[
 P_{2n}^{\rm C}(x)=T_n(x^2),\quad
 T_n(y)=(-1)^nR_n(-y),\quad 0\le n\le j,
\]
where $R_n$ denotes its monic para--Racah polynomials.

Let $\nu_t$ be its normalised weight at $-t^2$, given by
\eqref{eq:R-weights-full}--\eqref{eq:R-normalisation}.
For distinct lifts with $t\ne0$, the complementary weights are
\begin{equation}\label{eq:BI-masses}
 \omega^{\rm C}(t)=\frac{t+A}{2t}\nu_t,\quad
 \omega^{\rm C}(-t)=\frac{t-A}{2t}\nu_t.
\end{equation}
Their sum is $\nu_t$, and they satisfy the annihilation identity
proved above. These conditions determine the functional, since
every polynomial has the form
$p_0(x^2)+\left(x-A\right)p_1(x^2)$.
At $t=A$, the weight at $-A$ vanishes, leaving $2j+1$ nodes.

Put $E_{\rm C}=\supp\mathbf u^{\rm C}$ and
assume $\rho\notin E_{\rm C}$. Equation~\eqref{eq:BI-Geronimus}
then gives the Bannai weights:
\begin{equation}\label{eq:BI-partner-masses}
 \begin{aligned}
 \omega^{\rm B}(x)
 &=\frac{\kappa\omega^{\rm C}(x)}{x-\rho},
       \quad x\in E_{\rm C},\\[7pt]
 \omega^{\rm B}(\rho)
 &=1-\kappa\sum_{x\in E_{\rm C}}
              \frac{\omega^{\rm C}(x)}{x-\rho}.
 \end{aligned}
\end{equation}
Thus normalisation fixes the extra mass.

For odd $j$, the same construction uses
\[
 A=-\frac{b+j+1+a}{4},\quad
 C=\frac{a-b-j+1}{4},\quad \omega=\alpha.
\]
For $j=1$, use instead the two-point quadratic input
$\nu=(1-\alpha)\delta_{-A^2}+\alpha\delta_{-C^2}$;
no classicality assertion for this input is needed.
These are the other two published branches. The quadratic
recurrence identifies the complementary functional, and
\eqref{eq:BI-Geronimus} determines its Bannai partner.
Consequently, the weights and monic recurrences agree with
\cite{PVZ23}, including the free parameter and terminal degree.

\subsection{The identification}
We can now collect the identifications proved above.
\begin{theorem}
Under the regularity and distinctness assumptions of this section,
the full orthogonality functionals of the five para-families are classical
in the sense of Definition~\ref{def:classical}, for the maps and pairs in
Table~\ref{tab:structural-location}.
For each of the first four families, the two component functionals
span the solution space of its structural equation.
Every normalised solution is their linear combination with a free parameter
$\alpha$; para--Krawtchouk chooses $\alpha=1/2$. Only the
central recurrence coefficients depend on this parameter, as in
Proposition~\ref{prop:central-block}.

For complementary para--Bannai--Ito the pair is $(0,x-A)$.
For its Bannai partner the pair is $((x-\rho)(x-A),0)$.
Both use the alternating map and are obtained from the quadratic
input described above.
\end{theorem}
\begin{proof}
The pairs in the table have the classical degrees. In each of the
first four cases, direct substitution in Proposition~\ref{lem:one-progression}
gives the weights. The product and normalisation calculations
identify the full functional. Formula~\eqref{eq:four-dn-lambda}
and Proposition~\ref{prop:free-moment} give the assertion about
all solutions. Proposition~\ref{prop:central-block} gives the
full recurrence and its dependence on $\alpha$. The alternating assertions follow from
Proposition~\ref{prop:BI-classical} and the functional identities
\eqref{eq:BI-Geronimus}--\eqref{eq:BI-partner-masses}.
Finally, uniqueness of the monic orthogonal polynomials for a
regular functional gives equality of the polynomials and their
full recurrences.
\end{proof}

The identifications concern the full published functionals.
The map and structural pair are fixed while $\alpha$ varies.
For the first four, Proposition~\ref{prop:central-block} gives
persymmetry at $\alpha=1/2$. For each alternating system separately, write $F_n,G_n$ for the auxiliary
factors of \cite{PVZ23}. At $\alpha=1/2$ they satisfy $F_n=G_{M-n}$,
where $M=2j$ or $2j+1$, for either parity of $j$
\cite[Sections~3.1 and~4.1, Appendices~A--B]{PVZ23}.
Their monic recurrence entries in our notation are
$B_n=\rho-F_n-G_n$ and $C_n=F_{n-1}G_n$, so
$B_{M-n}=B_n$ and $C_{M+1-n}=C_n$.
With real nodes and positive weights these members therefore
give the persymmetric Jacobi matrices used next.

\section{An application: matrices with perfect state transfer}
\label{sec:PST}

We now use the functional construction for a different purpose.
The problem is to choose a real Jacobi matrix that transfers a
state from its first site to its last site at a prescribed time.
In quantum spin chains, the aim is to transport quantum information
between the ends with probability one, without active control
during the transfer. We refer interested readers to
\cite{VZPST12,GTVZ17} for the physical setting and its connection
with orthogonal polynomials.

Here we recall the spectral criterion and reconstruction results
from those papers, using our notation. We express the criterion
through the two finite functionals of
Section~\ref{sec:two-progressions}, then add the structural equation
to obtain classical examples. No classicality assumption is
needed for the transfer criterion itself.

\subsection{The functional characterisation}
Fix $N\ge2$ and $T>0$. Let $e_0,\ldots,e_N$
be the coordinate vectors and let
\[
 Je_n=\sqrt{C_{n+1}}e_{n+1}
                  +B_ne_n+\sqrt{C_n}e_{n-1},
 \quad 0\le n\le N,
\]
with the outside terms omitted. Here $B_n\in\mathbb R$ for
$0\le n\le N$ and $C_n>0$ for $1\le n\le N$. This is the symmetric Jacobi
matrix associated with \eqref{eq:TTRR}.
We call $J$ a $T$-PST matrix if, for some real $\theta$,
\begin{equation}\label{eq:PST-definition}
 e^{-iTJ}e_0=e^{i\theta}e_N.
\end{equation}
Here the sign in the exponential only fixes a time convention.
Put $Re_n=e_{N-n}$. Persymmetry means $RJR=J$,
or equivalently
\[
 B_{N-n}=B_n,\quad C_{N+1-n}=C_n.
\]
It is a condition on the entries of $J$, without a condition
on the eigenvalue gaps.

Write the eigenvalues of $J$ in increasing order as
$x_0<\cdots<x_N$. Its normalised functional is
\[
 \mathbf u=\sum_{s=0}^{N}\omega_s\delta_{x_s},\quad
 \omega_s>0,\quad \sum_{s=0}^{N}\omega_s=1.
\]
Here $\omega_s$ is the square of the first coordinate of a
normalised real eigenvector corresponding to $x_s$.

This is the standard correspondence between Jacobi matrices
and finite orthogonality functionals \cite{GTVZ17}.
The monic polynomials $P_0,\ldots,P_N$ orthogonal with respect
to $\mathbf u$ have non-zero norms and satisfy
\eqref{eq:TTRR}, with
\[
 B_n=J_{nn},\quad 0\le n\le N,\quad
 C_n=J_{n-1,n}^{\,2},\quad 1\le n\le N.
\]
Conversely, a normalised positive functional on $N+1$ distinct
real nodes determines its monic orthogonal polynomials and
their recurrence coefficients. These give a unique Jacobi
matrix with diagonal entries $B_n$ and positive off-diagonal
entries $\sqrt{C_n}$. Thus specifying the matrix is equivalent
to specifying the functional.

We now impose perfect transfer at time $T$. The following
theorem expresses this additional condition in terms of
$\mathbf u$ and describes it more specifically through two
finite component functionals. Define
\begin{equation}
 Q(x)=\prod_{s=0}^{N}(x-x_s),\quad
 Z_Q=\sum_{s=0}^{N}\frac1{|Q'(x_s)|},\quad g=\frac\pi T.
\end{equation}

\begin{theorem}\label{thm:PST-functional}
For $N\ge2$ and a real Jacobi matrix with positive off-diagonal
entries, the following conditions are equivalent.
\begin{enumerate}[label=\textup{(\roman*)}]
\item $J$ is a $T$-PST matrix.
\item $J$ is persymmetric and
\begin{equation}\label{eq:PST-gaps}
 x_{s+1}-x_s=g m_s,\quad
 m_s\in\{1,3,5,\ldots\},\quad 0\le s<N.
\end{equation}
\item Its normalised functional is
\[
 \mathbf u=\tfrac12\mathbf u^{(0)}+\tfrac12\mathbf u^{(1)},
\]
where
\[
 \mathbf u^{(0)}=\sum_{k=0}^{r}\omega_k^{(0)}\delta_{x_{2k}},
 \quad
 \mathbf u^{(1)}=\sum_{k=0}^{\ell}\omega_k^{(1)}\delta_{x_{2k+1}},
\]
have strictly positive weights, mass one, and the same moments
through degree $N-1$, and the nodes satisfy \eqref{eq:PST-gaps}.
The integers $r,\ell$ are those in \eqref{eq:r-ell}.
\end{enumerate}
For any fixed ordered nodes, the functionals in \textup{(iii)}
are unique. The weights of their equal combination are
\begin{equation}\label{eq:PST-weights}
 \omega_s=\frac1{Z_Q|Q'(x_s)|}.
\end{equation}
Consequently, the ordered spectra satisfying \eqref{eq:PST-gaps}
are in one-to-one correspondence with $T$-PST matrices of size $N+1$.
\end{theorem}
\begin{proof}
We recall the standard spectral argument
\cite[Sections~III--IV]{VZPST12}. Let $v_s$ be a normalised real
eigenvector at $x_s$. Perfect transfer gives
$e^{-iTx_s}v_s(0)=e^{i\theta}v_s(N)$.
Thus the spectral measures at the two ends coincide. Uniqueness
of the Jacobi matrix associated with the measure implies that
$J$ coincides with its reversal: it is persymmetric.
For such a matrix, the oscillation theorem gives
$v_s(N)=(-1)^{N+s}v_s(0)$, so transfer holds precisely when
$e^{-iTx_s}=e^{i\theta}(-1)^{N+s}$ for every $s$.
Consecutive phases give \eqref{eq:PST-gaps}, and that condition
conversely gives the required phases.

The same signs give
$P_N(x_s)=(-1)^{N+s}\sqrt{h_N}$, where
$h_N=C_1\cdots C_N$. The cardinal polynomial
$\ell_s(x)=Q(x)/((x-x_s)Q'(x_s))$ has leading coefficient
$1/Q'(x_s)$. Pairing it with $P_N$ gives
$\omega_sP_N(x_s)=h_N/Q'(x_s)$, which yields
\eqref{eq:PST-weights} after normalisation
\cite[Section~3]{GTVZ17}. We now prove its functional formulation.

Lagrange interpolation gives
\begin{equation}\label{eq:PST-interpolation}
 \sum_{s=0}^{N}\frac{p(x_s)}{Q'(x_s)}=0,
 \quad \deg p\le N-1.
\end{equation}
Indeed, the left side is the coefficient of $x^N$ in the
interpolation polynomial for $p$. Since
$\operatorname{sgn}Q'(x_s)=(-1)^{N-s}$, the weights
\eqref{eq:PST-weights} give
\[
 \sum_{s\text{ even}}\omega_s p(x_s)
 =\sum_{s\text{ odd}}\omega_s p(x_s),\quad \deg p\le N-1.
\]
Taking $p=1$ shows that each sum has mass $1/2$.
Multiplying each restriction by two therefore gives the two
functionals in \textup{(iii)}. Their uniqueness follows because the
matrix $(x_s^k)_{0\le k<N,\,0\le s\le N}$ has rank $N$:
its kernel is generated by $(1/Q'(x_s))_s$, and the two
mass-one conditions fix the scale.

Conversely, suppose these two functionals have the stated
moments. Proposition~\ref{prop:central-block} at $\alpha=1/2$
gives persymmetry directly. Adding \eqref{eq:PST-gaps} gives
\textup{(ii)} and proves the equivalence.
\end{proof}

Thus the equal combination and the gap condition solve
different parts of the problem. Without the gaps, the same
functional construction characterises persymmetric Jacobi
matrices. No structural equation, and no uniform or quadratic
shape of the nodes, is required for that statement.

\subsection{Constructing a matrix from its spectrum}
Theorem~\ref{thm:PST-functional} is a characterisation. To use
it, choose $x_0\in\mathbb R$ and positive odd integers
$m_0,\ldots,m_{N-1}$, then put
\[
 x_s=x_0+g\sum_{k=0}^{s-1}m_k,\quad 1\le s\le N.
\]
Construct $Q$, compute \eqref{eq:PST-weights}, and recover
the monic recurrence by finite orthogonalisation. This gives
every $T$-PST matrix with positive off-diagonal entries.
The resulting orthogonal polynomials need not be classical.

The algorithms in \cite[Section~3]{GTVZ17} reduce the amount
of orthogonalisation. To state one of them in our notation, put
\[
 \Pi_0(x)=\prod_{k=0}^{r}(x-x_{2k}),\quad
 \Pi_1(x)=\prod_{k=0}^{\ell}(x-x_{2k+1}),\quad
 s_i=\sum_{x\in\supp\mathbf u^{(i)}}x.
\]
For $N=2j+1$, the polynomials $P_j,P_{j+1}$ and the central
coefficient are
\begin{equation}
 P_{j+1}=\frac{\Pi_0+\Pi_1}{2},\quad
 P_j=\frac{\Pi_0-\Pi_1}{s_1-s_0},\quad
 C_{j+1}=\frac{(s_1-s_0)^2}{4}.
\end{equation}
For $N=2j$, they are
\begin{equation}
 P_j=\Pi_1,\quad B_j=s_0-s_1,\quad
 P_{j+1}=\frac{\Pi_0+(x-B_j)\Pi_1}{2}.
\end{equation}
These are \cite[Lemma~3.3]{GTVZ17}, rather than new recurrence
formulas. Divide $P_{n+1}$ by $P_n$. The quotient is $x-B_n$
and the remainder is $-C_nP_{n-1}$. The negative of its leading
coefficient gives $C_n$; division by $-C_n$ gives $P_{n-1}$. Continuing down to $P_0=1$
recovers the first half. Reflection gives the other half.
The weighted restrictions to alternate nodes also allow the
shorter orthogonalisation described in that paper.

\subsection{Adding classicality}
To obtain transfer matrices associated with classical orthogonal
polynomials, choose the nodes through an admissible map and
obtain the weights from the structural equation instead of from
\eqref{eq:PST-weights}.

\begin{corollary}\label{cor:classical-PST}
Fix $N\ge2$ and $T>0$. Suppose one admissible map $X$ and one
pair $(\phi,\psi)$ satisfy Proposition~\ref{lem:one-progression}
on the two progressions in \eqref{eq:two-components}, with
$r,\ell$ as in \eqref{eq:r-ell}. Suppose the nodes are real
and distinct, both normalised component functionals have
strictly positive weights, and
\[
 d_{N-1}=0,\quad d_n\ne0,\quad n\ne N-1.
\]
If the ordered union of the nodes satisfies \eqref{eq:PST-gaps},
then the Jacobi matrix of
$\tfrac12\mathbf u^{(0)}+\tfrac12\mathbf u^{(1)}$ is $T$-PST.
Its functional is classical with respect to the same map
and structural pair.
\end{corollary}
\begin{proof}
Proposition~\ref{prop:free-moment} gives equality of the
component moments through degree $N-1$. In the ordered union,
write $\varepsilon_s=-1$ on the first support and $\varepsilon_s=1$
on the second, and let $\nu_s>0$ be the corresponding component
weight. The moment equality gives
\[
 \sum_{s=0}^{N}\varepsilon_s\nu_sx_s^k=0,
 \quad 0\le k<N.
\]
The Vandermonde matrix has a one-dimensional kernel, generated by
$(1/Q'(x_s))_{s=0}^{N}$. Its entries alternate in sign; positivity
of $\nu_s$ therefore forces the supports to alternate in the
ordered union. They are the two alternate-node supports in
Theorem~\ref{thm:PST-functional}, after labelling.
The structural equation is linear,
so the combination solves \eqref{eq:structural}. All its
weights are positive, hence it is regular of order $N+1$.
Theorem~\ref{thm:PST-functional} gives the transfer property.
\end{proof}

This gives a procedure. Fix an admissible
map and two finite progressions. Compute the ordered gaps
and choose parameters that make them odd multiples of $g$.
Choose $\phi,\psi$ of degrees at most 2 and 1, respectively,
so that $\lambda_-$ vanishes at the starts and $\lambda_+$ at the ends.
Check the interior factors, positive weight ratios, and the condition on
$d_n$. Proposition~\ref{lem:one-progression} then gives the
two functionals. The standard classical formulas give the
recurrence in the common degree ranges, and
Proposition~\ref{prop:central-block} gives the remaining coefficients
and their reflection. For these candidates there is no
need to derive a new recurrence theory. We provide one example in the quadratic setting.

\begin{example}[A quadratic construction]
Fix $j\ge1$ and $T>0$, and put $N=2j+1$ and $g=\pi/T$.
Consider
\[
 X(s)=\kappa s^2+\eta s,\quad \kappa>0,\quad \eta\in\mathbb R.
\]
Let the progressions start at $a$ and $c=a+\delta$, with
$0<\delta<1$. Write
\[
 h=\frac{\eta}{2\kappa},\quad
 u=a+h>0,\quad v=u+\delta,\quad
 x_*=-\frac{\eta^2}{4\kappa}.
\]
Then $X(s)=\kappa\left(s+h\right)^2+x_*$, on the admissible domain
$U=\mathbb C\setminus\{-h+n/2:n\in\mathbb Z\}$.

We first choose the endpoint factors. With $\xi=s+h$, take
\[
 \begin{aligned}
 \lambda_-(s)
 &=-\frac{\kappa}{4j\xi}
   \left(\xi-u\right)\left(\xi-v\right)
   \left(\xi+u+j\right)\left(\xi+v+j\right),\\[7pt]
 \lambda_+(s)
 &=-\frac{\kappa}{4j\xi}
   \left(\xi+u\right)\left(\xi+v\right)
   \left(\xi-u-j\right)\left(\xi-v-j\right).
 \end{aligned}
\]
These close the progressions at $a+j$ and $c+j$.
Reconstructing the pair from \eqref{eq:lambdas}, and putting
\[
 H=u\left(u+j\right)+v\left(v+j\right),\quad
 K=uv\left(u+j\right)\left(v+j\right),
\]
gives
\begin{equation}\label{eq:quadratic-PST-pair}
 \begin{aligned}
 \phi(x)
 &=-\frac{\left(x-x_*\right)^2
          +\kappa\left(j^2-H\right)\left(x-x_*\right)
          +\kappa^2K}{2j},\\[7pt]
 \psi(x)&=x-x_*-\frac{\kappa H}{2}.
 \end{aligned}
\end{equation}

The interior weight ratios are positive, so
Proposition~\ref{lem:one-progression} gives two normalised
positive finite functionals. Moreover, $d_n=1-n/(2j)$.
Proposition~\ref{prop:free-moment} therefore gives agreement
of their moments through degree $N-1$. Choosing $\alpha=1/2$
now gives a persymmetric Jacobi matrix by
Proposition~\ref{prop:central-block}.

We still have to choose the parameters for transfer at time $T$.
The nodes $x_{2k}=X(a+k)$ and $x_{2k+1}=X(c+k)$, $0\le k\le j$,
interlace, and their gaps are
\begin{equation}\label{eq:quadratic-PST-gaps}
 \begin{aligned}
 x_{2k+1}-x_{2k}
 &=\kappa\delta\left(2u+2k+\delta\right),
       &&0\le k\le j,\\[7pt]
 x_{2k+2}-x_{2k+1}
 &=\kappa\left(1-\delta\right)
             \left(2u+2k+1+\delta\right),
       &&0\le k<j.
 \end{aligned}
\end{equation}
Whenever these are odd multiples of $g$, the matrix is
$T$-PST by Corollary~\ref{cor:classical-PST}.

The recurrence can be written explicitly without computing
moments. Put $\Sigma=u+v$ and $\theta_n=j+1-n$.
The standard quadratic formulas, together with
Proposition~\ref{prop:central-block}, give
\begin{equation}\label{eq:quadratic-PST-recurrence}
 \begin{aligned}
 B_n
 &=x_*+\frac{\kappa}{2}
       \left(H+n\left(N-n\right)\right),\\[7pt]
 C_n
 &=\frac{\kappa^2n\left(N+1-n\right)
          \left(\Sigma+n-1\right)\left(\Sigma+N-n\right)
          \left(\theta_n^2-\delta^2\right)}
         {4\left(4\theta_n^2-1\right)}.
 \end{aligned}
\end{equation}
The ranges are $0\le n\le N$ and $1\le n\le N$, respectively.
Both reflection identities are visible in these expressions. To obtain
transfer matrices with both map coefficients $\kappa,\eta$ non-zero, take
\[
 \kappa=8Mg,\quad
 \delta=\frac14,\quad
 u=\frac{4L+1}{8},\quad
 a=u-\frac{\eta}{2\kappa},\quad c=a+\frac14,
\]
where $M$ is any positive odd integer, $L$ is any non-negative
integer, and $\eta\in\mathbb R\setminus\{0\}$ is arbitrary. The nodes are
\[
 \begin{aligned}
 x_{2k}&=x_*+\frac{Mg}{8}\left(8k+4L+1\right)^2,\\[7pt]
 x_{2k+1}&=x_*+\frac{Mg}{8}\left(8k+4L+3\right)^2.
 \end{aligned}
\]
Their gaps divided by $g$ are
\[
 M\left(4k+2L+1\right),\quad
 M\left(12k+6L+9\right),
\]
respectively. These are positive odd integers, so every such
choice gives a classical $T$-PST matrix. The parameter $\eta$
remains free and changes only the common spectral shift $x_*$.

For a concrete matrix, take $j=1$, $L=0$, $M=1$ and $\eta=4g$.
Then
\[
 X(s)=8gs^2+4gs,\quad a=-\frac18,\quad c=\frac18,
\]
and the ordered nodes are $-3g/8$, $5g/8$, $77g/8$ and $117g/8$.
Formula \eqref{eq:quadratic-PST-recurrence} gives
\[
 \begin{aligned}
 B_0=B_3&=\frac{17g}{8},&
 B_1=B_2&=\frac{81g}{8},\\[7pt]
 C_1=C_3&=\frac{75g^2}{4},&
 C_2&=9g^2.
 \end{aligned}
\]
Thus
\[
 J=\frac g8
 \begin{pmatrix}
 17&20\sqrt3&0&0\\[7pt]
 20\sqrt3&81&24&0\\[7pt]
 0&24&81&20\sqrt3\\[7pt]
 0&0&20\sqrt3&17
 \end{pmatrix}.
\]
Its spectral gaps are $g,9g,5g$, giving perfect transfer
at $T=\pi/g$.
\end{example}

\section{Manufacturing a family: the structure behind the example}
\label{sec:manufactured}

A further example makes the distinction between producing formulas
and explaining a structure explicit. It illustrates the recurrence
construction from components with common moments in
Proposition~\ref{prop:central-block}.
We choose the radii freely. Common moments fix the positive weights
and the central isospectral deformation of \cite[Section~4]{GTVZ17},
derived here directly from the component functionals.

Take $0<a<b<c$ and $0<\alpha<1$. On $\{\pm\sqrt a,\pm\sqrt c\}$,
assign total masses $p,1-p$ to the two pairs; on
$\{0,\pm\sqrt b\}$, assign masses $1-q,q$ to zero and the pair.
Split each pair equally. All odd moments then vanish, and agreement
of the second and fourth moments requires
\[
 \begin{aligned}
 pa+(1-p)c&=qb,\\[7pt]
 pa^2+(1-p)c^2&=qb^2.
 \end{aligned}
\]
The determinant is $-b(c-a)(a+c-b)\ne0$, so these equations
determine the weights uniquely. Put, locally in this section,
\[
 d=a+c-b,\quad M=\frac{ac}{d},\quad
 p=\frac{c(c-b)}{(c-a)d},\quad q=\frac{ac}{bd}.
\]
Then $d>0$, $0<p,q<1$, since
\[
 1-p=\frac{a(b-a)}{(c-a)d},\quad
 1-q=\frac{(b-a)(c-b)}{bd}.
\]
Define the mass-one functionals
\[
 \begin{aligned}
 \mathbf u^{(0)}&=\frac p2(\delta_{\sqrt a}+\delta_{-\sqrt a})
       +\frac{1-p}{2}(\delta_{\sqrt c}+\delta_{-\sqrt c}),\\[7pt]
 \mathbf u^{(1)}&=(1-q)\delta_0
       +\frac q2(\delta_{\sqrt b}+\delta_{-\sqrt b}),\\[7pt]
 \mathbf u_\alpha&=(1-\alpha)\mathbf u^{(0)}+\alpha\mathbf u^{(1)}.
 \end{aligned}
\]
The components have positive weights on 4 and 3 distinct points,
respectively. They are regular of orders 4 and 3. Their odd moments
vanish, and both have second moment $M$ and fourth moment $bM$.
Thus their moments agree through degree 5, whereas
\[
 \langle\mathbf u^{(0)}-\mathbf u^{(1)},x^6\rangle
 =\frac{ac(b-a)(c-b)}{d}>0.
\]
The seven mixture weights are positive for every $0<\alpha<1$.
A non-zero real polynomial of degree at most 6 cannot vanish at all
7 nodes, so its squared norm is positive and $\mathbf u_\alpha$
is regular of order 7.

These components meet the hypotheses of
Proposition~\ref{prop:central-block}, with $N=6$, $j=3$,
$r=3$, $\ell=2$. The common polynomials begin with
\[
 P_0=1,\quad P_1=x,\quad P_2=x^2-M,\quad P_3=x(x^2-b).
\]
The smaller component terminates at $P_3$; the larger one has
\[
 \Pi_0=(x^2-a)(x^2-c)=xP_3-dP_2.
\]
The proposition therefore gives the entire recurrence:
\begin{equation}\label{eq:manufactured-seven-recurrence}
 \begin{gathered}
 B_n=0,\quad 0\le n\le6,\\[7pt]
 C_1=C_6=M,\quad C_2=C_5=b-M=\frac{(b-a)(c-b)}d,\\[7pt]
 C_3=(1-\alpha)d,\quad C_4=\alpha d.
 \end{gathered}
\end{equation}
Every $C_n$ is positive. Only the two central coefficients depend on
$\alpha$, and the terminal polynomial is
\[
 P_7=x(x^2-a)(x^2-b)(x^2-c).
\]
Consequently all these Jacobi matrices have the same seven-point
spectrum; $\alpha=1/2$ gives the persymmetric member.

Their classicality is global. Use the alternating admissible map
$X(s)=e^{\pi is}$ on the half-step-invariant domain $U=\mathbb C$,
as in Section~\ref{sec:para-BI}. The polynomial
$Sf(x)=(f(ix)+f(-ix))/2$ is even, so symmetry gives
\[
 \langle\mathbf S(x\mathbf u_\alpha),f\rangle
 =\langle\mathbf u_\alpha,xSf\rangle=0,
 \quad f\in\mathcal P.
\]
Thus $(\phi,\psi)=(0,x)$ is the same structural pair on every
polynomial degree. The atom at zero is legitimate: the calculus
acts on all polynomials by polynomial extension, although
$0\notin X(U)$. Every symmetric functional satisfies this alternating
equation; it leaves all even moments free. The example therefore
represents the recurrence construction from common moments;
it does not satisfy the hypotheses of Proposition~\ref{prop:free-moment}. Its
additional content is explicit: the chosen supports and matched
moments determine positive weights and the entire recurrence.

For instance, $a=2$, $b=5$, $c=7$ gives $p=q=7/10$ and
\[
 (C_1,\ldots,C_6)
 =(7/2,3/2,4(1-\alpha),4\alpha,3/2,7/2).
\]
The size seven is not essential. For any integer $j\ge2$, choose
$0<r_1<\cdots<r_j$, put $N=2j$, and arrange the nodes
$0,\pm r_1,\ldots,\pm r_j$ in increasing order.
Use \eqref{eq:PST-weights} with
$Q(x)=x\prod_{k=1}^j(x^2-r_k^2)$.
The Lagrange identity \eqref{eq:PST-interpolation} gives mass $1/2$
to each alternate-node restriction and common moments through degree
$2j-1$. After normalisation, label the positive components $\mathbf u^{(0)}$
and $\mathbf u^{(1)}$ on $j+1$ and $j$ nodes, respectively. Both are symmetric: $Q'$ is even and reflection
preserves index parity. Their combinations use the same alternating pair $(0,x)$;
Proposition~\ref{prop:central-block} shows that only $C_j,C_{j+1}$
depend on $\alpha$. For the larger component,
\eqref{eq:central-node-sums-even} gives
$c^{(0)}=\sum_{k=1}^j(-1)^{j-k}r_k^2>0$; positivity follows by
pairing successive squared radii. Hence
\[
 C_j=(1-\alpha)c^{(0)},\quad C_{j+1}=\alpha c^{(0)}.
\]
Nine, eleven and arbitrarily many points therefore give the same
construction.

The radii can be changed indefinitely while retaining the domain,
admissible map and global equation. Each choice permits another family
name and further identities within this same framework. Any additional
algebraic or physical distinction must be established separately.
Examples are informative when they test a conjecture, reveal an
obstruction or distinguish structures, as the counterexamples below do.
Repetition of the same construction must justify what it adds.

\section{Four recent preprints}\label{sec:discussion}

We examine the first arXiv versions of
\cite{BC26,BCLMRV26K,BCLMRV26R,BCLMRV26B}. The first treats
para--Krawtchouk, para--Racah, $q$-para--Racah and para--Bannai--Ito,
including the complementary branch. The three companions study,
respectively, the Hamiltonians and continuum limits of para--Krawtchouk,
para--Racah and para--Bannai--Ito. The counterexamples below refute
specific assertions in the first three papers. The Scarf companion
enters the structural comparison and is credited for its explicit
orthogonality condition in the probability interpretation of revival.
Following \cite{BC26}, we use $\sigma$ for the mass parameter and
$u_n$ for the monic recurrence coefficient denoted $C_n$ here.

On distinct nodes the support decomposition is valid when the
operators descend to the finite function space and preserve its two
component spaces. This function space is equivalently the polynomial
quotient by the terminal polynomial. The announced hypotheses do
not always ensure these conditions. Several further assertions are false.
We give counterexamples under the stated hypotheses and distinguish
their consequences from the confluences that remain valid.

\paragraph{\textbf{A false irreducibility theorem in the positive range.}}
Theorem~4 of \cite{BC26} is false. Take the complementary
para--Bannai--Ito family with
\[
 N=4,\quad j=2,\quad a=-4,\quad b=0,\quad
 \sigma=\tfrac12,\quad\beta=1.
\]
The original recurrence \cite[Eqs.~(3.3)--(3.5)]{PVZ23}
has $C_1=\cdots=C_4=1/2$ in our notation, and its nodes are
\[
 (x_0,x_1,x_2,x_3,x_4)
 =(5/4,-3/4,-5/4,3/4,1/4).
\]
All its weights are positive and the nodes are distinct.
Equation~(214) of \cite{BC26} gives the roots $\pm3/4$,
so the theorem declares the summands reducible. On the odd nodes, multiplication by $x$ and its difference
operator $\mathcal D_\beta$ \cite[Eqs.~(185)--(191)]{BC26} at $\beta=1$
give, in node order $(x_1,x_3)$,
\[
 X_o=\diag(-3/4,3/4),\quad
 (\mathcal D_1)_o=
 \begin{pmatrix}2/3&-2/3\\[7pt]-1/3&1/3\end{pmatrix}.
\]
The only invariant lines of $X_o$ are the coordinate lines;
neither is invariant under $\mathcal D_1$. This summand is
irreducible. The false criterion counts roots of the undeformed
reflection coefficient $C(x)$ before including the actual coefficient
$C(x)-\beta(x-\rho_2)/(2x)$, where $\rho_2=1/4$, and ignores other connections.
It therefore fails for a genuine positive orthogonality system,
not merely at a singular parameter value. 

There is also a counterexample to the general Proposition~4 on
which that criterion rests. On the Krawtchouk grid $0,\ldots,5$
with $p=1/2$ and weights $w_x=2^{-5}\binom{5}{x}$, define
$T_{\pm2}f(x)=f(x\pm2)$ and let
\[
 L=\tfrac12(x-4)(x-5)(T_{+2}-I)
   +\tfrac12x(x-1)(T_{-2}-I),\quad Y=L^2+14L.
\]
The Krawtchouk polynomials diagonalise $L$, with eigenvalues
$2n(n-5)$, and hence diagonalise $Y$. Both operators preserve
the two progressions and close at their ends. On the even progression $0,2,4$, multiplication by $x$ and $Y$
are represented by
\[
 X_e=\diag(0,2,4),\quad
 Y_e=\begin{pmatrix}
 -30&0&30\\[7pt]0&-12&12\\[7pt]6&24&-30
 \end{pmatrix}.
\]
The coefficient of the interior shift $2\mapsto0$ vanishes,
while the connections $0\leftrightarrow4\leftrightarrow2$
make the pair irreducible. The proof of Proposition~4 asserts
tridiagonality, which Proposition~3 neither assumes nor establishes.
A zero coefficient disconnects a nearest-neighbour path; it need not
disconnect a higher-order difference operator. The asserted general
criterion is false.

\paragraph{\textbf{A deformation erased from its own recurrence.}}
Equations~(83)--(84) of \cite{BC26} contain no $\sigma$;
they give the equal-mass para--Racah system while claiming to
retain $\sigma=e_1/(e_1+e_2)$, its mass parameter.
This is disproved in the positive
interior by
\[
 N=3,\quad a=1,\quad c=\tfrac32,\quad\sigma=\tfrac13.
\]
The Wilson truncation prescribed there, with $e_1=1,e_2=2$,
or the original recurrence \cite{LVZ16}, gives
\[
 B_1=-\tfrac{79}{24},\quad B_2=-\tfrac{107}{24},\quad
 C_2=\tfrac{49}{18}.
\]
The displayed formulas instead give
\[
 B_1=B_2=-\tfrac{31}{8},\quad C_2=\tfrac{49}{16}.
\]
This discrepancy is independent of polynomial normalisation.
If $\widehat R_n=k_nR_n$ with $k_n\ne0$, the coefficient of
$\widehat R_n$ in $x\widehat R_n$ is still $B_n$.
The product of the coefficients joining indices $n-1,n$ is
$(k_{n-1}/k_n)(C_nk_n/k_{n-1})=C_n$.
Its own Wilson factors \cite[Eqs.~(256)--(257)]{BC26} give
$A_1=-7/6$, $C_2^{\rm W}=-7/3$, hence
$C_2=A_1C_2^{\rm W}=49/18$ in the monic recurrence.
Scaling the orthogonality functional also changes neither coefficient.
Thus they do not define the announced family. The discrepancy
also changes the physical symmetry: the genuine central fields
differ, whereas the printed matrix is persymmetric. The parameter
controls these entries, as Proposition~\ref{prop:central-block}
explains; it cannot be discarded while retaining the identification.

\paragraph{\textbf{An invalid positivity domain and a failed finite evaluation representation.}}
Before Eq.~(140), \cite{BC26} states that positivity of all $u_n$
restricts the parameters to the listed domain. That restriction
excludes positive systems of arbitrarily large size. For any integer
$j\ge2$, take
\[
 N=2j+1,\quad q=\tfrac12,\quad a=2^{j+1},\quad
 c=\tfrac34,\quad\sigma=\tfrac12.
\]
In Eq.~(136), the two factors containing $ac$ have the same sign
for $n\le j$ and $n\ge j+2$. The factors
$aq^n-cq^{j+1}$ and $cq^n-aq^{j+1}$ have opposite signs,
as do $q^n-1$ and $q^n-q^{2j+2}$. The denominator is positive
away from $n=j+1$, where the separate formula is a positive square.
Thus every $u_n>0$; the real diagonal entries give a positive
finite orthogonality system. Its nodes are distinct: equality
between the two progressions would require $3$ to be an integral
power of $2$. Yet
\[
 a/c=2^{j+3}/3>2,\quad
 1<ac=3\cdot2^{j-1}<2^{2j}=q^{1-N},
\]
contradicting both listed restrictions. Replacing $a$ by
$(aq^j)^{-1}=1/2$ reverses the first progression and gives the
same recurrence and functional. The alleged restriction excludes
a positive system merely because its support is parametrised in
the opposite direction.
For $j=2$, its positive coefficients are
\[
 (u_1,\ldots,u_5)
 =\left(\tfrac{625}{4032},\tfrac{3965}{8064},
 \tfrac{41209}{36864},\tfrac{3965}{8064},\tfrac{625}{4032}\right).
\]

Conversely, the displayed inequalities do not ensure positivity
for arbitrarily large even $N$. For every integer $j\ge1$, put
\[
 N=2j,\quad q=\tfrac12,\quad a=\tfrac12,\quad
 c=\tfrac34,\quad\sigma=\tfrac12,\quad\tau=q^j.
\]
All the inequalities hold and the $N+1$ nodes are distinct, but
Eq.~(138) gives
\[
 u_j=u_{j+1}
 =-\frac{(1-\tau)(2-\tau)(8-3\tau)(4-3\tau)}{432\tau^2}<0.
\]
The prescribed couplings $\sqrt{u_n}$ are not real, so this
recurrence does not give a real Hermitian Jacobi or $XX$-chain
Hamiltonian. A real spectrum does not make these prescribed
recurrence coefficients positive.

Theorem~3 of \cite{BC26} uses Eq.~(140) for both parities
as its parameter hypotheses. Those conditions can fail even to
define the announced finite evaluation representation in arbitrarily
large dimension. For every $j\ge2$, take
\[
 N=2j,\quad q=\tfrac14,\quad a=2\cdot4^{j-1},\quad
 c=\tfrac98a,\quad\sigma=\tfrac12.
\]
All the inequalities hold: $a/c=8/9$ and
$ac=(9/8)q^{1-N}>q^{1-N}$. The last two nodes of the first
progression coincide, since $aq^{j-1}=2$ and $aq^j=1/2$
give the same value $5/4$ of $x(w)=(w+w^{-1})/2$.
All other support points are distinct. Let $F$ be the monic product
over these $2j$ distinct points. It is the zero function on the
entire grid. In Eqs.~(144)--(145) the other Askey--Wilson parameters
are $b=2$ and $d=4/9$. Its forward coefficient is
\[
 A(w)=\frac{(1-aw)(1-bw)(1-cw)(1-dw)}{(1-w^2)(1-qw^2)}.
\]
Here the Askey--Wilson operator is
\[
 \begin{aligned}
 (Yf)(x(w))={}&A(w)\{f(x(qw))-f(x(w))\}\\[7pt]
 &+A(w^{-1})\{f(x(q^{-1}w))-f(x(w))\},\quad f\in\mathcal P.
 \end{aligned}
\]
The residue and the derivative of the shifted variable are
\[
 \operatorname*{Res}_{w=2}A(w)=-\frac{(1-2a)(1-2c)}9,\quad
 \left.\frac{d}{dw}\{x(qw)-x(w)\}\right|_{w=2}=-\frac34.
\]
The backward coefficient $A(w^{-1})$ is regular and vanishes there
because $1-b/w=0$. Since $x(2q)=x(2)=5/4$, the product of
the residue and derivative gives the removable limit of the full expression:
\[
 (YF)(5/4)=\frac{(1-2a)(1-2c)}{12}F'(5/4)\ne0.
\]
The kernel of evaluation is not invariant: identical value data
have different images. A representation using derivatives at the
repeated node would be a different object; it would not justify
the stated diagonal evaluation representation of dimension $N+1$
or positive orthogonality on these support points.

\paragraph{\textbf{Classical components whose orthogonality was incorrectly denied.}}
Consider the odd truncation $N=2j+1$, $j\ge1$.
The Krawtchouk companion \cite[Section~4, after Eq.~(33)]{BCLMRV26K}
asserts that its dual Hahn parameters lie outside the range of
orthogonality. The cited reference explicitly includes the range
in which both parameters are less than minus the component truncation index
\cite[Eq.~(9.6.2), p.~209]{KLS10}. Here those parameters are
$-j-1-\gamma/2$ and $-j-1+\gamma/2$; both are less than $-j$
for every $0<\gamma<2$. The component grid is
$\lambda(k)=k(k-N)$, $0\le k\le j$.
For $j=1$, $\gamma=1/2$, its even component has
\[
 R_0=1,\quad R_1(\lambda)=1+\tfrac45\lambda,\quad
 \mathbf v=\tfrac38\delta_0+\tfrac58\delta_{-2}.
\]
Direct summation gives
$\langle\mathbf v,R_1\rangle=0$ and
$\langle\mathbf v,R_1^2\rangle=3/5>0$.
Indeed, its recurrence factors $A_t,C_t$
\cite[Eqs.~(39)--(40)]{BCLMRV26K} give
$A_{t-1}C_t>0$ for $1\le t\le j$ in both components when $0<\gamma<2$.
Both are positive finite dual Hahn systems on their own grid.

The Racah companion repeats the same false assertion
\cite[Section~4, after Eq.~(37)]{BCLMRV26R}.
For $j=2$, $a=1/2$, $c=1$, its even component has parameters
$(-3,-3,1/2,-1/2)$ and the positive functional
\[
 \mathbf v=\frac{66\delta_0+165\delta_2+25\delta_6}{256}.
\]
For its hypergeometric normalisation,
\[
 R_0=1,\quad R_1=1-\tfrac8{15}\lambda,\quad
 R_2=1-\tfrac45\lambda+\tfrac{16}{75}\lambda(\lambda-2),
\]
and
$\langle\mathbf v,R_nR_m\rangle
=\diag(1,11/15,33/50)_{n,m}$.
Negative parameters do not preclude a positive finite measure;
the displayed measures refute the assertions. In fact the Racah
components in Eq.~(37), with $N=2j+1$ and $j\ge1$, remain
positive throughout the announced range
$a>-1/2$, $|a|<c<a+1$. Put $d=c-a$, $S=a+c$.
On the even grid $\lambda(s)=s(s+2a)$ their difference coefficients are
\[
 \begin{aligned}
 B(s)&=\frac{(s-j)(s-j-d)(s+S)(s+2a)}
                  {(2s+2a)(2s+2a+1)},\\[7pt]
 D(s)&=\frac{s(s+2a+j)(s+S+j)(s-d)}
                  {(2s+2a-1)(2s+2a)}.
 \end{aligned}
\]
Since $0<d<1$ and $S>0$, every interior edge is positive:
$B(s)>0$ for $0\le s<j$ and $D(s)>0$ for $1\le s\le j$.
At $s=0$, evaluate the full boundary difference expression at
removable singularities. Its effective coefficient to $\lambda(1)$ is
$j(j+d)S/(2a+1)>0$. For $a=0$, both shifted arguments equal
$\lambda(1)=1$, so their coefficients add. For $a=1/2$, the backward
argument equals $\lambda(0)$ and its difference contribution vanishes.
In the finite operator and the weight ratio below, $B(0)$ denotes
this effective boundary coefficient.
Choose $v_0>0$ and define $v_{s+1}/v_s=B(s)/D(s+1)$.
These positive weights make the operator self-adjoint; its distinct polynomial
eigenfunctions are orthogonal. For the odd component, exchanging
$a,c$ replaces $d$ by $-d$, with $c>0$ and $0<d<1$;
both interior coefficients remain positive, with the same boundary
convention. Thus both papers deny
the positive classical orthogonality of the very components used
in their odd-$N$ constructions, throughout the stated ranges.
Those assertions obscure precisely the classical structure used
to obtain the limits. The dual
Hahn--Laguerre and Racah--Jacobi confluences remain valid;
they do not require the claimed loss of orthogonality.

\paragraph{\textbf{What constitutes structural novelty.}}
The counterexamples above invalidate particular theorems,
parameter domains, recurrences and finite evaluation representations.
They cannot be answered by another family name, another singular
truncation or another application of the same spectral data.
Nor does correcting them turn a support parametrisation into a
new classical structure. The conclusion of \cite{BC26} itself
acknowledges that these are not new Askey-scheme families and
identifies their classical components. That acknowledgement
does not repair its false criteria or the companions' denials
of positive component orthogonality. The first four calculi are those of \cite{NU83,NU84}, and the
alternating calculus is developed in \cite{C26,CG26b}.
The components are classical polynomial expressions. A decomposition
of the finite evaluation representation is valid when the operators
descend on distinct nodes and preserve the two component spaces.
This gives an explanation at the level of finite representations;
it does not establish the global structural equation governing
the full functionals.

The finite families, their singular truncations and isospectral
deformations predate these preprints
\cite{VZ12,LVZ16,LVZ18,BGVZ21,PVZ23,GTVZ17}.
Algebraic realisations also appear in \cite{TVZ17,GTVZ20,BGVZ22};
\cite[Section~4]{TVZ17} already describes a reducible
$q$-para--Krawtchouk representation. The first preprint seeks a unified
identification of the component modules; the companions study continuum
limits. Those analyses go beyond the finite structural identity.
The principal finite para-families studied in these preprints are not new.

The present construction identifies the full functionals through
a global structural equation, retaining the free moment and mass
deformation in the four non-alternating cases. For the transfer matrices,
regularity, positivity and the spectral gap condition are checked separately. For $N=2j+1$ and fixed parameters in their positive ranges,
the dual Hahn--Laguerre and Racah--Jacobi confluences at fixed degree,
away from the singular points and with the explicit mode normalisations
in \cite[Section~4]{BCLMRV26K} and \cite[Section~4]{BCLMRV26R},
remain valid. The finite revival identities in
\cite[Sections~12 and~11, respectively]{BCLMRV26K,BCLMRV26R}
are also valid. Their coefficients are amplitudes; their squared moduli
give probabilities for the original and reflected packets when those
packets are orthogonal, as \cite[Section~9]{BCLMRV26B} makes explicit.
They do not repair the false statements reviewed above.

Under the hypotheses of Section~\ref{sec:para-BI}, the discrete
para--Bannai--Ito family and its complementary branch used in the
Scarf companion \cite[Sections~2 and~6]{BCLMRV26B} are covered by
the present framework, including their full functionals, weights,
free parameter and complete monic recurrences, for both parities of $j$.
The continuum analysis connects these finite systems to the Scarf
systems and little $-1$ Jacobi functions already studied in \cite{PVZ11}.
Its additional contribution is that confluence and its operator
interpretation; the underlying Scarf Hamiltonian was already known.
These limits require separate arguments and are not refuted by the
counterexamples above.

None of these preprints presents experimental validation of its new
results or establishes a new quantitative advantage in robustness or
practical performance. Mathematical correctness is a minimum requirement.
An additional construction must identify the problem it solves,
what it explains beyond the existing theory and the mathematical
advance it achieves. Claims of practical usefulness must identify
a concrete physical problem and demonstrate an advantage over existing
methods; an exactly solvable Hamiltonian does not establish such
usefulness. A Hamiltonian interpretation cannot repair invalid positivity
hypotheses or an operator that does not descend to the announced finite
space. The errors above require correction before the affected formulas
can support physical conclusions.

\section{Conclusion}\label{sec:conclusion}
Under the stated hypotheses, the full orthogonality functionals of the
five para-families are classical in the sense of
Definition~\ref{def:classical}. The four non-alternating
systems arise from two finite solutions of one global equation.
Their common moments, relative mass and reflection determine the
complete recurrence. The alternating construction identifies the
complementary and Bannai partners. Equal positive combinations in
the first four cases give persymmetric Jacobi matrices; the ordered
spectral gaps decide transfer at a prescribed time.

These are classical families in disguise in the structural sense
proved here. We have specialised the broader theory of \cite{C26}
to expose their construction; that theory also supplies further
families, their functionals and finite measures. The seven-point
example makes the consequence concrete: the catalogue can grow
without a new governing principle. The counterexamples expose failures in the stated positivity
domains, irreducibility criteria, recurrence formulas and finite
evaluation representations, and show when the prescribed couplings
fail to define a Hermitian model.

In an era of AI-assisted symbolic production, mathematical publication
must establish what a result explains and adds. Errors must be
corrected, and additional examples must earn their place through
mathematical or physical content. Rewarding manufacture as a
substitute for understanding must stop. \emph{Enough, please.}

\appendix
\section{A classical dictionary: substitutions and resonant limits}
\label{app:classical-identifications}

Using the notation of \cite{KLS10}, we record resonant limits of
classical expressions for the four full non-alternating families,
classical polynomial expressions for their finite components, and quadratic Wilson
expressions for the alternating family. The limits retain the relative
mass parameter; these correspondences do not replace the global
structural identities proved above. Classicality in
Definition~\ref{def:classical} concerns the full functionals; a component
with fewer than three masses need not have the required order.
We retain the source names for identification; separate names for
instances of the same construction are unnecessary here.

Set $r=\lfloor N/2\rfloor$, $\ell=\lfloor(N-1)/2\rfloor$, so
$N=r+\ell+1$. Use the parameters and the regularity and distinctness assumptions
of Section~\ref{sec:families}, with $\alpha\ne0,1$.
Here $t$ is an auxiliary limit parameter.
Write $\mathsf M_x F$ for the rational continuation of $F$ divided
by its leading coefficient in $x$, after the displayed variable changes,
cancelling common parameter factors before evaluation.
In every limit this normalisation is performed before $t\to0$;
convergence means convergence of all polynomial coefficients.
The expressions in \cite{KLS10} are used in this algebraic sense.
The finite orthogonality measures are the functionals identified
in Section~\ref{sec:families}.
The superscripts $\rm PK,PR,qPR,qPK$ denote the four monic families
identified in Section~\ref{sec:families}; their degree-$N+1$ polynomials
are the terminal support polynomials.

\Needspace{7\baselineskip}
\subsection{The full non-alternating families}

\paragraph{\textbf{Para--Krawtchouk: continuous Hahn.}}
Let $H_n(z;a,b,c,d)$ be the continuous Hahn polynomial in
\cite[Section~9.4, Eq.~(9.4.1), p.~200]{KLS10}. Then
\begin{equation}\label{eq:KLS-PK-limit}
 \begin{aligned}
 P_n^{\rm PK}(x)
 &=\lim_{t\to0}\mathsf M_x H_n
       \bigl(ix/2;0,\gamma/2+(1-\alpha)t,
                    -r+\alpha t,-\ell-\gamma/2\bigr),\\[7pt]
 &\quad 0\le n\le N+1.
 \end{aligned}
\end{equation}
The published para--Krawtchouk system chooses $\alpha=1/2$;
other values give the mass deformation in this article.

\paragraph{\textbf{Para--Racah: Wilson.}}
Let $W_n(z;a,b,c,d)$ denote the Wilson polynomial in its squared
variable $z$, as in
\cite[Section~9.1, Eq.~(9.1.1), p.~185]{KLS10}. The published
truncation \cite{LVZ16} is exactly
\begin{equation}\label{eq:KLS-PR-limit}
 \begin{aligned}
 P_n^{\rm PR}(x)
 &=\lim_{t\to0}\mathsf M_x W_n
       \bigl(x;a,-a-r+\alpha t,c,-c-\ell+(1-\alpha)t\bigr),\\[7pt]
 &\quad 0\le n\le N+1.
 \end{aligned}
\end{equation}
In both identities the sum of the four parameters is $1-N+t$.
The limiting prescription specifies the central entries at resonance;
a bare substitution at $t=0$ does not specify the full family.

\paragraph{\textbf{$q$-para--Racah: Askey--Wilson.}}
For the Askey--Wilson polynomial $p_n$ of
\cite[Section~14.1, Eq.~(14.1.1), p.~415]{KLS10},
\begin{equation}\label{eq:KLS-qPR-limit}
 \begin{aligned}
 P_n^{q\rm PR}(x)
 &=\lim_{t\to0}\mathsf M_x p_n
    \bigl(x;a,a^{-1}q^{-r+\alpha t},
               c,c^{-1}q^{-\ell+(1-\alpha)t}\mid q\bigr),\\[7pt]
 &\quad 0\le n\le N+1.
 \end{aligned}
\end{equation}
This is the truncation of \cite{LVZ18}; the product of its four
parameters is $q^{1-N+t}$.

\paragraph{\textbf{$q$-para--Krawtchouk: big $q$-Jacobi.}}
Use the variable and parameter $\Delta$ of
\eqref{eq:qK-weights-full}, put $L=\Delta q^r$, and let $J_n$ be
the big $q$-Jacobi polynomial of
\cite[Section~14.5, Eq.~(14.5.1), p.~438]{KLS10}. Then
\begin{equation}\label{eq:KLS-qPK-limit}
 \begin{aligned}
 P_n^{q\rm PK}(x)
 &=\lim_{t\to0}\mathsf M_x J_n
       \bigl(x/L;A_t,B_t,C\mid q\bigr),
       \quad 0\le n\le N+1,\\[7pt]
 A_t&=q^{-r-1}(1+\alpha t),\quad
 B_t=q^{-\ell-1}\bigl(1+(1-\alpha)t\bigr),\quad
 C=\Delta^{-1}q^{-r-1}.
 \end{aligned}
\end{equation}
Here $A_tB_t\to q^{-N-1}$. To verify the continuous Hahn and big
$q$-Jacobi limits, use their monic recurrences
\cite[Eqs.~(9.4.3)--(9.4.4), pp.~200--201, and
(14.5.3)--(14.5.4), pp.~438--439]{KLS10}.
For the substitutions above, write $\mathcal A_n,\mathcal C_n$ for
their recurrence factors. The variable changes give, respectively,
\[
 \begin{aligned}
 (B_n,C_n)&=\bigl(2(\mathcal A_n+\mathcal C_n),
                         4\mathcal A_{n-1}\mathcal C_n\bigr),\\[7pt]
 (B_n,C_n)&=\bigl(L(1-\mathcal A_n-\mathcal C_n),
                         L^2\mathcal A_{n-1}\mathcal C_n\bigr).
 \end{aligned}
\]
Cancelling the common simple zero retains the proportions
$\alpha,1-\alpha$. The limits of
$(\mathcal A_j,\mathcal C_{j+1})$ for odd $N$, and of
$(\mathcal A_j,\mathcal C_j)$ for even $N$, are
$(\alpha,1-\alpha)$ times
\[
 \begin{array}{c|cc}
  &\text{continuous Hahn}&\text{big }q\text{-Jacobi}\\[7pt]
  N=2j+1&(j+1)\gamma/2&
    -(1-\Delta)(1-q^{j+1})/[L(1-q)]\\[7pt]
  N=2j&j(1-\gamma/2)&(1-q^j)(1-\Delta)/[L(1-q)].
 \end{array}
\]
The lower-degree defining sums give the component polynomials in
\eqref{eq:KLS-PK-components} and \eqref{eq:KLS-qPK-components},
within their stated ranges. Substitution at the remaining indices gives
at $t=0$ the identities $\mathcal A_{N-n}=\mathcal C_n$ and
$\mathcal C_{N-n}=\mathcal A_n$, hence the reflections in
Proposition~\ref{prop:central-block}; the displayed factors give its
centres. Finally $\mathcal A_N\to0$, and induction
in the monic recurrence includes $P_{N+1}=\Pi_0\Pi_1$.
Thus the limits hold coefficient by coefficient through the terminal
degree. The Wilson and Askey--Wilson limits are the published
truncations cited above; in all four cases the direction of approach
retains the parameter that a singular substitution would discard.

\subsection{The finite components as classical polynomial expressions}

Write $Q_n^{(0)},Q_n^{(1)}$ for the monic orthogonal polynomials
of the separate component functionals, with $0\le n\le r$ and
$0\le n\le\ell$, respectively. These are component systems;
mixture polynomials of larger degree require the full construction.

\paragraph{\textbf{Para--Krawtchouk components: Hahn.}}
For the Hahn polynomial $Q_n(z;A,B,R)$ of
\cite[Section~9.5, Eq.~(9.5.1), p.~204]{KLS10}, set
$A=-\ell-1-\gamma/2$, $B=-r-1+\gamma/2$. Then
\begin{equation}\label{eq:KLS-PK-components}
 \begin{aligned}
 Q_n^{(0)}(x)&=\mathsf M_x Q_n(x/2;A,B,r),\\[7pt]
 Q_n^{(1)}(x)&=\mathsf M_x Q_n((x-\gamma)/2;B,A,\ell).
 \end{aligned}
\end{equation}
For $0<\gamma<2$, these parameters lie in the positive orthogonality
range with negative parameters explicitly included in
\cite[Eq.~(9.5.2)]{KLS10}.

\paragraph{\textbf{Para--Racah components: Racah.}}
For $R_n(z;A,B,G,D)$ of
\cite[Section~9.2, Eqs.~(9.2.1)--(9.2.2), p.~190]{KLS10},
\begin{equation}\label{eq:KLS-PR-components}
 Q_n^{(0)}(x)=\mathsf M_x R_n
      (-x-a^2;-r-1,-\ell-1,a+c-1,a-c).
\end{equation}
Interchanging $(a,r)$ and $(c,\ell)$ gives $Q_n^{(1)}$.
The standard Racah variable becomes $k(k+2a)=-x-a^2$ on
$x=-(a+k)^2$.

\paragraph{\textbf{$q$-para--Racah components: $q$-Racah.}}
For $R_n(z;A,B,G,D\mid q)$ of
\cite[Section~14.2, Eqs.~(14.2.1)--(14.2.2), p.~422]{KLS10},
\begin{equation}\label{eq:KLS-qPR-components}
 Q_n^{(0)}(x)=\mathsf M_x R_n
       (2ax;q^{-r-1},q^{-\ell-1},ac/q,a/c\mid q).
\end{equation}
Again interchange $(a,r)$ and $(c,\ell)$ for the second component.
The standard variable becomes $q^{-k}+a^2q^k=2ax$.

\paragraph{\textbf{$q$-para--Krawtchouk components: $q$-Hahn.}}
For $Q_n(z;A,B,R\mid q)$ of
\cite[Section~14.6, Eq.~(14.6.1), p.~445]{KLS10}, put
$A=\Delta^{-1}q^{-r-1}$, $B=\Delta q^{-\ell-1}$,
$L=\Delta q^r$. Then
\begin{equation}\label{eq:KLS-qPK-components}
 \begin{aligned}
 Q_n^{(0)}(x)&=\mathsf M_x Q_n(x/L;A,B,r\mid q),\\[7pt]
 Q_n^{(1)}(x)&=\mathsf M_x Q_n(x/q^\ell;B,A,\ell\mid q).
 \end{aligned}
\end{equation}
Reversing the node index within each progression gives the
nodes $q^{-k}$ of \cite[Eq.~(14.6.2)]{KLS10}; the polynomial variables undergo exactly
the displayed dilations.

\subsection{The alternating family: quadratic Wilson expressions}

The alternating reconstruction of \cite[Proposition~1.3]{CG26b}
and the Geronimus inversion explain these families;
Section~\ref{sec:para-BI} proves their finite form.
Their correspondence with the Wilson family in \cite{KLS10} is quadratic.
Use the complementary family of Section~\ref{sec:para-BI} and the
parameter paths below, following \cite[Eqs.~(2.1)--(2.3)]{PVZ23}:
\begin{equation}\label{eq:KLS-pBI-Wilson}
 \begin{aligned}
 P_{2n}^{\rm C}(x)
 &=\lim_{t\to0}\mathsf M_x W_n
       (-x^2;\rho_2(t),\rho_1(t)+1,
                  \tfrac12-r_1(t),\tfrac12-r_2(t)),\\[7pt]
 P_{2n+1}^{\rm C}(x)
 &=(x-A)\lim_{t\to0}\mathsf M_x W_n
       (-x^2;\rho_2(t)+1,\rho_1(t)+1,
                  \tfrac12-r_1(t),\tfrac12-r_2(t)),
 \end{aligned}
\end{equation}
for degrees through the complementary terminal degree $2j+1$.
Put $\rho=(b-j-1+a)/4$ and use the following substitutions
\cite[Eqs.~(3.1)--(3.2) and~(A.1)]{PVZ23}:
\[
 \begin{array}{ll}
 j\text{ even}:&
 \rho_1(t)=\rho,\quad \rho_2(t)=A=(b-j-1-a)/4,\\[7pt]
 &r_1(t)=\rho+(j+1)/2-\alpha t,\\[7pt]
 &r_2(t)=A+(j+1)/2-(1-\alpha)t;\\[7pt]
 j\text{ odd}:&
 \rho_1(t)=\rho,\quad r_1(t)=-(b-j-1-a)/4,\\[7pt]
 &\rho_2(t)=-(j+1)/2-\rho+\alpha t,\\[7pt]
 &r_2(t)=(j+1)/2-r_1(t)-(1-\alpha)t,
 \end{array}
\]
with $A=\rho_2(0)$ in the odd case. The Bannai partner follows
from the Geronimus relation \eqref{eq:BI-Geronimus} and the full
weights \eqref{eq:BI-partner-masses}. Section~\ref{sec:para-BI}
identifies the quadratic input and proves the global alternating
structural equation. These expressions locate the family within
the common framework without claiming an ordinary affine Racah
identification for the full alternating system.

\needspace{14\baselineskip}
\section*{Acknowledgements}
The authors acknowledge financial support from the Centre for
Mathematics of the University of Coimbra (CMUC), funded by the
Portuguese Foundation for Science and Technology (FCT), under the
projects UID/00324/2025 and UID/PRR/00324/2025. The first author
acknowledges financial support from the FCT under grant
2022.00143.\allowbreak{}CEECIND/\allowbreak{}CP1714/\allowbreak{}CT0002.
The second author acknowledges financial support from the FCT under
grant UI.BD.\allowbreak154694.2023.

ChatGPT~6 was used for an additional check of suspected errors in
Section~\ref{sec:discussion}, and helped refine several counterexamples
to make their scope explicit. The mathematical development of the paper
predates this use. Responsibility for the arguments and conclusions
remains with the authors.

For the authors, the support acknowledged here means the salary of an
assistant researcher and the funding of a young doctoral student.
Our modest remuneration does not diminish our commitment to scientific
rigour and social responsibility.

\section*{Conflict of interest}
The authors declare that they have no conflict of interest.

\end{document}